\documentclass[11pt,reqno]{amsart}

\usepackage{amsmath,amsfonts,amsthm,amssymb}
\usepackage{color,soul}
\usepackage{graphicx}
\usepackage{verbatim}
\usepackage{color}
\usepackage{amscd}
\usepackage{verbatim}
\usepackage{mathrsfs}
\usepackage[colorlinks,citecolor=red,pagebackref,hypertexnames=false, breaklinks]{hyperref}
\usepackage{esint}
\usepackage[normalem]{ulem}
\usepackage{cancel}

\usepackage[frenchb, english]{babel}

\begin{document}
\newcommand{\M}{\mathscr{M}}
\newcommand{\loc}{{\mathrm{loc}}}
\newcommand{\core}{C_0^{\infty}(\Omega)}
\newcommand{\sob}{W^{1,p}(\Omega)}
\newcommand{\sobloc}{W^{1,p}_{\mathrm{loc}}(\Omega)}
\newcommand{\merhav}{{\mathcal D}^{1,p}}
\newcommand{\be}{\begin{equation}}
\newcommand{\ee}{\end{equation}}
\newcommand{\mysection}[1]{\section{#1}\setcounter{equation}{0}}
%%%%%%%%%%%%
\newcommand{\laplace}{\Delta}%{\triangle}
\newcommand{\pl}{\laplace_p}
\newcommand{\grad}{\nabla}%{\bigtriangledown}
\newcommand{\pd}{\partial}
\newcommand{\bo}{\pd}
\newcommand{\csub}{\subset \subset}
\newcommand{\sm}{\setminus}
\newcommand{\ssm}{:}
\newcommand{\diver}{\mathrm{div}\,}
%%%%%%%%%%%%%%%
\newcommand{\bea}{\begin{eqnarray}}
\newcommand{\eea}{\end{eqnarray}}
\newcommand{\bean}{\begin{eqnarray*}}
\newcommand{\eean}{\end{eqnarray*}}
\newcommand{\thkl}{\rule[-.5mm]{.3mm}{3mm}}
%%%%%%%%%%%%%%%%%%%%%%%%%%%
\newcommand{\cw}{\stackrel{\rightharpoonup}{\rightharpoonup}}
\newcommand{\id}{\operatorname{id}}
\newcommand{\supp}{\operatorname{supp}}
\newcommand{\calE}{\mathcal{E}}
\newcommand{\calF}{\mathcal{F}}
\newcommand{\wlim}{\mbox{ w-lim }}
\newcommand{\mymu}{{x_N^{-p_*}}}
\newcommand{\R}{{\mathbb R}}
\newcommand{\N}{{\mathbb N}}
\newcommand{\Z}{{\mathbb Z}}
\newcommand{\Q}{{\mathbb Q}}
\newcommand{\abs}[1]{\lvert#1\rvert}
%%%%%%%%%%%
\newtheorem{theorem}{Theorem}[section]
\newtheorem{corollary}[theorem]{Corollary}
\newtheorem{lemma}[theorem]{Lemma}
\newtheorem{notation}[theorem]{Notation}
\newtheorem{definition}[theorem]{Definition}
\newtheorem{remark}[theorem]{Remark}
\newtheorem{proposition}[theorem]{Proposition}
\newtheorem{assertion}[theorem]{Assertion}
\newtheorem{problem}[theorem]{Problem}
%%%%%%%%%%%%%%%%%%
\newtheorem{conjecture}[theorem]{Conjecture}
\newtheorem{question}[theorem]{Question}
\newtheorem{example}[theorem]{Example}
\newtheorem{Thm}[theorem]{Theorem}
\newtheorem{Lem}[theorem]{Lemma}
\newtheorem{Pro}[theorem]{Proposition}
\newtheorem{Def}[theorem]{Definition}
\newtheorem{Exa}[theorem]{Example}
\newtheorem{Exs}[theorem]{Examples}
\newtheorem{Rems}[theorem]{Remarks}
\newtheorem{Rem}[theorem]{Remark}

\newtheorem{Cor}[theorem]{Corollary}
\newtheorem{Conj}[theorem]{Conjecture}
\newtheorem{Prob}[theorem]{Problem}
\newtheorem{Ques}[theorem]{Question}
\newtheorem*{corollary*}{Corollary}
\newtheorem*{theorem*}{Theorem}
\newcommand{\pf}{\noindent \mbox{{\bf Proof}: }}

%%%%%%%%%%%%%%%%%%
%\newenvironment{proof}{{\bf Proof.}}{\hfill $\bowtie$\vskip4mm}

\renewcommand{\theequation}{\thesection.\arabic{equation}}
\catcode`@=11 \@addtoreset{equation}{section} \catcode`@=12
%%%%%%%%%%%%%%%%%%%%%%
\newcommand{\Real}{\mathbb{R}}
\newcommand{\real}{\mathbb{R}}
\newcommand{\Nat}{\mathbb{N}}
\newcommand{\ZZ}{\mathbb{Z}}
\newcommand{\CC}{\mathbb{C}}
\newcommand{\Pess}{\opname{Pess}}
\newcommand{\Proof}{\mbox{\noindent {\bf Proof} \hspace{2mm}}}
\newcommand{\mbinom}[2]{\left (\!\!{\renewcommand{\arraystretch}{0.5}
\mbox{$\begin{array}[c]{c}  #1\\ #2  \end{array}$}}\!\! \right )}
\newcommand{\brang}[1]{\langle #1 \rangle}
\newcommand{\vstrut}[1]{\rule{0mm}{#1mm}}
\newcommand{\rec}[1]{\frac{1}{#1}}
\newcommand{\set}[1]{\{#1\}}
\newcommand{\dist}[2]{$\mbox{\rm dist}\,(#1,#2)$}
\newcommand{\opname}[1]{\mbox{\rm #1}\,}
\newcommand{\mb}[1]{\;\mbox{ #1 }\;}
\newcommand{\undersym}[2]
 {{\renewcommand{\arraystretch}{0.5}  \mbox{$\begin{array}[t]{c}
 #1\\ #2  \end{array}$}}}
\newlength{\wex}  \newlength{\hex}
\newcommand{\understack}[3]{%
 \settowidth{\wex}{\mbox{$#3$}} \settoheight{\hex}{\mbox{$#1$}}
 \hspace{\wex}  \raisebox{-1.2\hex}{\makebox[-\wex][c]{$#2$}}
 \makebox[\wex][c]{$#1$}   }%
%%Macros for changing font size in math.
\newcommand{\smit}[1]{\mbox{\small \it #1}}% only for letters, numbers
\newcommand{\lgit}[1]{\mbox{\large \it #1}}% only for letters, numbers
\newcommand{\scts}[1]{\scriptstyle #1}
\newcommand{\scss}[1]{\scriptscriptstyle #1}
\newcommand{\txts}[1]{\textstyle #1}
\newcommand{\dsps}[1]{\displaystyle #1}
%%%%%%%%%%%%%%%%%%%%%%%%%%%%%%%%%%%%%%%%%%%%%%
\newcommand{\dx}{\,\mathrm{d}x}
\newcommand{\dy}{\,\mathrm{d}y}
\newcommand{\dz}{\,\mathrm{d}z}
\newcommand{\dt}{\,\mathrm{d}t}
\newcommand{\dr}{\,\mathrm{d}r}
\newcommand{\du}{\,\mathrm{d}u}
\newcommand{\dv}{\,\mathrm{d}v}
\newcommand{\dV}{\,\mathrm{d}V}
\newcommand{\ds}{\,\mathrm{d}s}
\newcommand{\dS}{\,\mathrm{d}S}
\newcommand{\dk}{\,\mathrm{d}k}

\newcommand{\dphi}{\,\mathrm{d}\phi}
\newcommand{\dtau}{\,\mathrm{d}\tau}
\newcommand{\dxi}{\,\mathrm{d}\xi}
\newcommand{\deta}{\,\mathrm{d}\eta}
\newcommand{\dsigma}{\,\mathrm{d}\sigma}
\newcommand{\dtheta}{\,\mathrm{d}\theta}
\newcommand{\dnu}{\,\mathrm{d}\nu}
\newcommand{\Ker}{\mathrm{Ker}}
\newcommand{\Ima}{\mathrm{Im}}
\newcommand{\Lip}{\mathrm{Lip}}

\def\Xint#1{\mathchoice
{\XXint\displaystyle\textstyle{#1}}%
{\XXint\textstyle\scriptstyle{#1}}%
{\XXint\scriptstyle\scriptscriptstyle{#1}}%
{\XXint\scriptscriptstyle\scriptscriptstyle{#1}}%
\!\int}
\def\XXint#1#2#3{{\setbox0=\hbox{$#1{#2#3}{\int}$ }
\vcenter{\hbox{$#2#3$ }}\kern-.6\wd0}}
\def\dashint{\Xint-}

\newcommand{\Rd}{\color{red}}
\newcommand{\Bk}{\color{black}}
\newcommand{\Mg}{\color{magenta}}
\newcommand{\Wh}{\color{white}}
\newcommand{\Bl}{\color{blue}}
\newcommand{\Yl}{\color{yellow}}

%%%%%%%%%%%%%%%%%%%%%%%%%%%%%%%Macros for Greek letters.
%%%%%%%%%%%%%%%%%%%%%%%%%%%%%%%%%%%%%%%%%%%%%%%%%%%%%%%%

\renewcommand{\div}{\mathrm{div}}
\newcommand{\red}[1]{{\color{red} #1}}

\newcommand{\cqfd}{\begin{flushright}                  
			 $\Box$
                 \end{flushright}}
                 
\newcommand{\todo}[1]{\vspace{5 mm}\par \noindent
\marginpar{\textsc{}} \framebox{\begin{minipage}[c]{0.95
\textwidth} \tt #1
\end{minipage}}\vspace{5 mm}\par}

%\begin{titlepage}

\title[Reverse inequality for Riesz transforms] {Reverse inequality for the Riesz transforms on Riemannian manifolds}
%\title[Short title]{Reverse inequality for the Riesz transform}
\date{04-09-2026}
\author{Baptiste Devyver and Emmanuel Russ}
\address{Baptiste Devyver, Institut Fourier - Universit\'e de Grenoble Alpes, France}
\email{baptiste.devyver@univ-grenoble-alpes.fr}
\address{Emmanuel Russ, Institut de Math\'ematiques de Marseille - Aix-Marseille Universit\'e, France}
\email{emmanuel.russ@univ-amu.fr}

\maketitle
\tableofcontents

\selectlanguage{french} 
\begin{abstract} 
Soit $M$ une vari\'et\'e riemannienne compl\`ete v\'erifiant la condition de doublement pour le volume des boules g\'eod\'esiques et des in\'egalit\'es de Poincar\'e $L^q$ \`a l'\'echelle pour une famille de boules \'eloign\'ees de l'origine, pour un $q<2$.  Nous montrons l'in\'egalit\'e $\left\Vert \Delta^{1/2}f\right\Vert_p\lesssim \left\Vert \nabla f\right\Vert_p$ pour tout $p\in ]q,2]$, g\'en\'eralisant des r\'esultats d'Auscher et Coulhon. Notre conclusion s'applique, notamment, si $M$ poss\`ede un nombre fini d'extr\'emit\'es euclidiennes. Les preuves reposent de fa\c con importante sur des in\'egalit\'es de Hardy, \'egalement nouvelles dans ce contexte et qui pr\'esentent leur propre int\'er\^et.
\end{abstract}
\selectlanguage{english} 
\begin{abstract}  
Let $M$ be a complete Riemannian manifold satisfying the doubling volume condition for geodesic balls and $L^q$ scaled Poincar\'e inequalities on suitable remote balls for some $q<2$. We prove the inequality $\left\Vert \Delta^{1/2}f\right\Vert_p\lesssim \left\Vert \nabla f\right\Vert_p$ for all $p\in (q,2]$, which generalizes previous results due to Auscher and Coulhon. Our conclusion applies, in particular, when $M$ has a finite number of Euclidean ends. The proof strongly relies on Hardy inequalities, which are also new in this context and of independent interest.
\end{abstract}

\section{Introduction}
Throughout the paper, if $A(f)$ and $B(f)$ are two nonnegative quantities defined for all $f$ belonging to a set $E$, the notation $A(f)\lesssim B(f)$ means that there exists $C>0$ such that $A(f)\le CB(f)$ for all $f\in E$, while $A(f)\simeq B(f)$ means that $A(f)\lesssim B(f)$ and $B(f)\lesssim A(f)$.  \par
\noindent Let $M$ be a complete connected noncompact Riemannian manifold. Denote by $\mu$ the Riemannian measure, by $\nabla$ the Riemannian gradient and by $\Delta$ the Laplace-Beltrami operator. The volume of a geodesic ball $B$ will be denoted by $\mu(B)$ instead of $\mu(B)$. In this work, we consider the following three inequalities for $p\in (1,\infty)$ (where the $L^p$-norms are computed with respect to the measure $\mu$):

\begin{equation}\label{eq:Ep}\tag{$\mathrm{E}_p$}
||\Delta^{1/2}u||_p\lesssim ||\nabla u||_p\lesssim ||\Delta^{1/2}u||_p,\quad \forall u\in C_0^\infty(M)
\end{equation}

\begin{equation}\label{eq:Rp}\tag{$\mathrm{R}_p$}
||\nabla u||_p\lesssim ||\Delta^{1/2}u||_p,\quad \forall u\in C_0^\infty(M)
\end{equation}

\begin{equation}\label{eq:RRp}\tag{$\mathrm{RR}_p$}||\Delta^{1/2}u||_p\lesssim ||\nabla u||_p,\quad \forall u\in C_0^\infty(M)
\end{equation}
It follows easily from the Green formula and the self-adjointness of $\Delta$ that 

$$||\nabla u||_2^2=(\Delta u,u)=||\Delta^{1/2}||_2^2,\quad \forall u\in C_0^\infty(M).$$
Consequently, \eqref{eq:Ep} holds for $p=2$ on any complete Riemannian manifold. The inequality \eqref{eq:Rp} is equivalent to the $L^p$-boundedness of the Riesz transform $\mathscr{R}=\nabla\Delta^{-1/2}$. A well-known duality argument, originally introduced in \cite{B}, shows that \eqref{eq:Rp} implies ($\mathrm{RR}_q$) for $q=p'$ the conjugate exponent, but the converse implication does not hold (see Section \ref{sec:gauss} below). The present work focuses on the inequality \eqref{eq:RRp}, which we establish for suitable ranges of $p$ in situations where inequalities of the form $(\mathrm{R}_{p^{\prime}})$ do not hold. The proofs are strongly related to the geometry of the underlying manifold and to the behaviour of the heat kernel $p_t$, namely the kernel of the semigroup generated by $\Delta$. \par

\medskip

\noindent We consider the case where $p_t$ satisfies Gaussian type pointwise upper estimates, and prove that, if a scaled $L^q$ Poincar\'e inequality holds on remote balls of $M$ for some $q\in [1,2)$, then \eqref{eq:RRp} for $p\in (q,2]$ (see Section \ref{sec:gauss} below for precise statements). 
%The second situation we consider is provided by fractal-like manifolds where $p_t$ satisfies {\it sub-Gaussian} upper bounds involving a parameter $\beta>2$. The case of the Vicsek manifold, considered in \cite{CCFR}, shows that, even when a Poincar\'e inequality on all balls holds (with a scaling related to the geometry of the manifold), the inequality \eqref{eq:RRp} may be false {\it for all } $p\in (1,2)$. In the present work, we therefore consider variants of \eqref{eq:RRp} where the exponent $\frac 12$ is replaced by.other exponents related to the geometry of the manifold and the paramter $\beta$ (see Section \ref{subsec:subgauss} below for a more detailed presentation).
\subsection{Previously known reverse Riesz inequalities} \label{sec:gauss}
In \cite{AC}, P. Auscher and T. Coulhon have studied the inequality \eqref{eq:RRp}, and the relationship between \eqref{eq:Rp} and ($\mathrm{RR}_q$), $q=p'$. In order to recall some of their results, we need to introduce some geometric inequalities about $M$. \par \noindent For all $x\in M$ and all $r>0$, let $B(x,r)$ be the open geodesic ball with center $x$ and radius $r$ and set $V(x,r):=\mu(B(x,r))$. Say that the doubling volume property holds if and only if, for all $x\in M$ and all $r>0$,
\begin{equation} \label{eq:DV} \tag{D}
V(x,2r)\lesssim V(x,r).
\end{equation}
By iteration, this condition implies at once that there exists $D>0$ such that for all $x\in M$ and all $0<r<R$,

\begin{equation}\label{eq:VD}\tag{VD}
V(x,R)\lesssim \left(\frac{R}{r}\right)^DV(x,r).
\end{equation}
An easy consequence of \eqref{eq:DV} is that for every $0<r\leq R$, and for every $x,y\in M$ such that $d(x,y)\leq r$, one has

\begin{equation}\label{eq:eqvol}
V(x,R)\simeq V(y,R).
\end{equation}
We also consider a reverse doubling volume condition: there exists $\nu>0$ such that, for all $x\in M$ and all $0<r<R$,
\begin{equation} \label{eq:RDV}\tag{RD}
\left(\frac Rr\right)^{\nu}V(x,r) \lesssim V(x,R) .
\end{equation}
It is known that since $M$ is non-compact and connected, \eqref{eq:DV} implies \eqref{eq:RDV} for some $\nu>0$. \par
\noindent Let $p\in [1,\infty)$. We consider the scaled $L^p$ Poincar\'{e} inequality on balls, namely :

\begin{equation}\label{eq:Pp}\tag{$\mathrm{P}_p$}
||f-f_B||_{L^p(B)}\lesssim r||\nabla f||_{L^p(B)},\quad f\in C^\infty(B),\,\forall B=B(x,r)\subset M,
\end{equation}
where $f_B$ denotes the average of $f$ on $B$, that is $f_B:=\mu(B)^{-1}\int_B f$. \par

\medskip

\noindent Among other things, Auscher and Coulhon prove in \cite{AC} that 

\begin{itemize}

\item[(i)] if the Hodge projector onto exact $1$-forms $\Pi:=\mathscr{R}\mathscr{R}^*$ is $L^p$-bounded, then ($\mathrm{RR}_{p'}$) implies \eqref{eq:Rp},

\item[(ii)] if \eqref{eq:DV} holds, as well as $(\mathrm{P}_q)$ for some $q\in [1,2)$, then, \eqref{eq:RRp} holds for all $p\in (q,2)$.

\end{itemize}
As a consequence of (ii), one can see that the implication ($\mathrm{RR}_{p'}$) $\Rightarrow$ \eqref{eq:Rp} is false in general. Indeed, let $M$ be a complete Riemannian manifold of dimension $n\geq 3$, such that $M$ has only one end, and this end is asymptotically conic (see \cite{GH1}); assume furthermore that the first eigenvalue of the cross-section of the corresponding cone is strictly less than $n-1$. Then, according to \cite[Theorem 1.4]{GH1}, \eqref{eq:Rp} holds on $M$ if and only if $p\in (1,p^*)$, where $n<p^*<+\infty$ only depends on the first eigenvalue of the cross-section of the cone. However, such a manifold $M$ satisfies $(\mathrm{P}_1$), so that \eqref{eq:RRp} holds for all $p\in (1,2]$ according to (ii). \par
\noindent Observe that $(ii)$ above does not apply in the case where $M$ is, for instance, the connected sum of two copies of $\R^n$, since $(P_2)$ does not hold in this case (\cite[Appendix]{GI}). As far as the Riesz transforms are concerned, it was shown in \cite{CCH} that \eqref{eq:Rp} holds on $M$ if and only if $1<p<n$ if $n\ge 3$, and if and only if $1<p\le 2$ if $n=2$, which implies that \eqref{eq:RRp} holds when $p>\frac{n}{n-1}$. However, the validity of $(\mathrm{RR}_q)$ for $1<q\le \frac{n}{n-1}$ remained open in that case. \par

%Introducing the main inequalities:
%
%\medskip

\subsection{New results}
In the present work, we extend statement $(ii)$ above to the case where the $L^q$ Poincar\'e inequality only holds on some ``remote'' balls of $M$.\par
\noindent  Let us fix once and for all a point $o$ in $M$. For $x\in M$, we will denote $r(x)=d(x,o)$. We let $B_0=B(o,r_0)$, where $r_0>0$ is large enough and will be determined later. 

\begin{definition} \label{def:admballs}
Let $x\in M$ and $r>0$.
\begin{enumerate}
\item The ball $B(x,r)$ is called {\em remote} if $r\leq \frac{r(x)}{2}$.
\item The ball $B(x,r)$ is called {\em anchored} if $x=o$.
\item The ball $B(x,r)$ is called {\em admissible} if either $B$ is remote, or $B(x,r)$ is anchored and $r\leq r_0$.
\end{enumerate}
\end{definition}
In this article, instead of $L^p$ Poincar\'{e} inequalities \eqref{eq:Pp} for {\em all} balls of $M$, we will consider the following assumption that $L^p$ Poincar\'{e} inequalities hold only for certain balls:

\begin{definition}
{\em 

We say that the $L^p$ Poincar\'{e} inequality holds {\em in the ends} of $M$ if, for every {\em admissible} balls $B$,

\begin{equation}\label{eq:PEp}\tag{$\mathrm{P}^E_p$}
||f-f_B||_{L^p(B)}\lesssim r||\nabla f||_{L^p(B)},\quad f\in C^\infty(B)
\end{equation}
where $r$ stands for the radius of $B$.

}
\end{definition}

Definition \ref{eq:PEp} generalizes to the case $p\in (1,\infty)$ the $L^2$ Poincar\'e inequality on ends introduced in \cite[Definition 1.4]{J}. It follows from \cite[Theorem 2.1]{HK} and the H\"older inequality that if $p\leq q$, \eqref{eq:PEp} $\Rightarrow$ ($\mathrm{P}^E_q$). See also the beginning of \cite[Section 4]{HK}. 
Let us also point out that if the Ricci curvature has a quadratic lower bound of the form:

\begin{equation} \label{QD} \tag{QD}
\mathrm{Ric}_x\geq -\frac{g}{1+r(x)^2},
\end{equation}
where $g$ is the Riemannian metric on $M$, then \eqref{eq:PEp} holds for all $p\geq1$ (this was first established by Buser in \cite{Bu}, a more elementary approach can be found in \cite[Theorem 5.6.5]{Saloffbook}). In particular, \eqref{eq:PEp} holds for all $p\geq1$ in the case where $M$ is the connected sum of two copies of $\R^n$. \par
\noindent Before stating our main theorem, we need to introduce the heat kernel $p_t(x,y)$, which is the kernel of the heat semigroup $e^{-t\Delta}$. Say that $p_t$ satisfies pointwise Gaussian upper bounds if
\begin{equation}\label{eq:UE}\tag{UE}
p_t(x,y)\lesssim \frac{1}{V(x,\sqrt{t})}\exp\left(-\frac{d^2(x,y)}{ct}\right),\quad \forall t>0,\,\forall x,y\in M.
\end{equation}
It is well-known (see \cite[Theorem 4]{Da}) that \eqref{eq:UE} implies analogous estimates for the time-derivatives $\frac{\partial^n}{\partial t^n}$: for every $n\in\N$,

\begin{equation}\label{eq:dUE}
\left\vert \frac{\partial^n}{\partial t^n}p_t(x,y)\right\vert\lesssim \frac{1}{t^nV(x,\sqrt{t})}\exp\left(-\frac{d^2(x,y)}{ct}\right),\quad \forall t>0,\,\forall x,y\in M.
\end{equation}
Sometimes we will use a slightly different (but equivalent, under \eqref{eq:DV}) version of \eqref{eq:dUE}, which we record here:

\begin{equation}\label{eq:dUE2}
\left\vert \frac{\partial^n}{\partial t^n}p_t(x,y)\right\vert\lesssim \frac{1}{t^nV(y,\sqrt{t})}\exp\left(-\frac{d^2(x,y)}{ct}\right),\quad \forall t>0,\,\forall x,y\in M.
\end{equation}
(the constants that we call $c$ in \eqref{eq:dUE} and \eqref{eq:dUE2} not necessarily being the same). \par
\noindent Say that $M$ has a finite number of ends if there exists an integer $N\ge 1$ such that, for all $R>0$, $M\setminus B(o,R)$ has at most $N$ unbounded connected components. It is known (\cite[Section[2.4.1]{C}) that condition \eqref{eq:DV} implies that $M$ has a finite number of ends. \par
\noindent We also consider the following geometric condition. 
\begin{definition}
{\em 

We say that $(M,g)$ with a finite number of ends satisfies the Relative Connectedness in the Ends (RCE) condition, if there is a constant $\theta\in (0,1)$ such that for any point $x$ with $r(x)\geq1$, there is a continuous path $c:[0,1]\to M$ satisfying

\begin{itemize}

\item $c(0)=x$.

\item the length of $c$ is bounded by $\frac{r(x)}{\theta}$.

\item $c([0,1])\subset B(o,\theta^{-1}r(x))\setminus B(o,\theta r(x))$.

\item there is a geodesic ray $\gamma:[0,+\infty)\to M\setminus B(o,r(x))$ with $\gamma(0)=c(1)$.

\end{itemize}

}
\end{definition}
When $M$ only has one end, the (RCE) condition is nothing but the (RCA) condition introduced in \cite{GS}. Let us also recall (\cite[Theorem 2.4]{C}) that, if \eqref{QD} and (RCE) hold, as well as the volume comparison property, namely

\begin{equation}\label{eq:VC}\tag{VC}
V(o,R)\lesssim V\left(x,\frac R2\right)
\end{equation}
for all $R\ge 1$ and all $x\in \partial B(o,R)$, then the relative Faber-Krahn inequality holds, hence \eqref{eq:UE} and \eqref{eq:DV} hold. 
  
\noindent The main purpose of this article is to show the following result:

\begin{theorem}\label{thm:main}

Let $M$ be a complete Riemannian manifold satisfying \eqref{eq:DV}, \eqref{eq:UE}, \eqref{eq:RDV} for some $\nu>2$ and {\em (RCE)}. Assume that for some $q\in (1,2]$ such that $q<\nu$, the $L^q$ Poincar\'e inequalities in the ends {\em ($\mathrm{P}_q^E$)} hold. Then, for every $p\in [q,2)$, \eqref{eq:RRp} holds on $M$. 

\end{theorem}

\begin{Rem}
{\em
The assumption $q<\nu$ in the statement of Theorem \ref{thm:main} is only required for the validity of the Hardy inequality (Theorem \ref{thm:hardy} below), which turns out to be crucial in the argument.}
\end{Rem}
\begin{Rem}
{\em

The following two questions naturally show up, in light of Theorem \ref{thm:main}:

\begin{enumerate}

\item Is the assumption (RCE) in Theorem \ref{thm:main} really necessary?

\item We believe that the result should still hold for $1<\nu\leq 2$. Evidence for this is Corollary \ref{cor:main}, which asserts that it is indeed the case, under more restrictive geometric assumptions on $M$. However, we have not been able to prove it in full generality.

\end{enumerate}
}
\end{Rem}

\begin{remark}
{\em 

Let us compare Theorem \ref{thm:main} with \cite[Theorem 0.7]{AC}. Assume that \eqref{eq:DV}, ($\mathrm{P}_q$) for some $q\in [1,2]$ and \eqref{eq:RDV} for some $\nu>q$ hold. Then ($\mathrm{P}_2$) holds as well; together with \eqref{eq:DV}, it follows that \eqref{eq:UE} holds (see \cite[Theorem 4.2.6]{Saloffbook}). Moreover, \cite[Proposition 0.3]{Min} shows that the conjunction of \eqref{eq:DV}, ($\mathrm{P}_q$) and \eqref{eq:RDV} for some $\nu>q$ imply the (RCA) condition. Since it is clear that ($\mathrm{P}_q$) $\Rightarrow$ ($\mathrm{P}^E_q$), it follows that for every $p\in (q,2)$, \eqref{eq:RRp} holds on $M$. In other words, under the condition \eqref{eq:RDV} for some $\nu>q$, assumptions in Theorem \ref{thm:main} are weaker than those of \cite[Theorem 0.7]{AC}. \par
\noindent Note also that, in Theorem \ref{thm:main}, the conclusion \eqref{eq:RRp} holds for all $p\in [q,2)$, while the corresponding conclusion in \cite[Theorem 0.7]{AC} under the assumption that  ($\mathrm{P}_q$) holds, is only stated for $p\in (q,2)$ (actually, a weak form of \eqref{eq:RRp} is proved for $p=q$ is proved in \cite[Section 1.2]{AC}). However, when ($\mathrm{P}_q$) holds, there exists $\varepsilon>0$ such that ($\mathrm{P}_{q-\varepsilon}$) is also satisfied (\cite[Theorem 1.0.1]{K}), so that \cite[Theorem 0.7]{AC} yields \eqref{eq:RRp} for $p=q$. 

}
\end{remark}
For manifolds with a quadratic decay of the Ricci curvature, we obtain the following result, which in the case $\nu>2$ is merely a corollary of Theorem \ref{thm:main}:

\begin{Cor}\label{cor:main}

Let $M$ be a complete Riemannian manifold satisfying \eqref{QD}, \eqref{eq:VC}, {\em (RCE)} and \eqref{eq:RDV} with $\nu>1$. Then, for every $p\in (1,+\infty)$, \eqref{eq:RRp} holds on $M$. In particular, ($RR_p$) holds for every $p\in (1,\infty)$ for $M$ a connected sum of two copies of $\R^n$, $n\geq 2$.

\end{Cor} 

\begin{proof}

First, we explain how, under the extra assumption that $\nu>2$, this follows from Theorem \ref{thm:main}. Indeed\Bk, we are going to show that the assumptions of Theorem \ref{thm:main} are satisfied with $q=1$. The $L^1$ Poincar\'e inequality in the ends follows from \cite[Theorem 5.6.5]{Saloffbook}. Now, as we have mentioned before, \eqref{QD}, \eqref{eq:VC} and (RCE) imply \eqref{eq:DV} and \eqref{eq:UE}. Thus, if $\nu>2$, the assumptions of Theorem \ref{thm:main} are satisfied, and therefore we get \eqref{eq:RRp} for all $p\in (1,2)$. The reverse inequalities \eqref{eq:RRp} for $p\in [2,+\infty)$ follow from \cite{CD} and the implication \eqref{eq:Rp} $\Rightarrow$ ($\mathrm{RR}_q$), $q=p'$.

The case $1<\nu\leq 2$ requires a different argument. It is dealt with later in the paper, more precisely it is a consequence of Lemma \ref{lem:Uterm} and Proposition \ref{pro:T_quadratic} (which in fact also give the result for any $\nu>1$).

\end{proof}

\noindent The proof of Theorem \ref{thm:main} relies on three major ingredients: the first one (Proposition \ref{prop:goodcovering} below) is the covering of $M$ by admissible balls $(B_{\alpha})_{\alpha\in \N}$ and the existence of an associated smooth partition of unity $(\chi_{\alpha})_{\alpha\in \N}$. The second one (Theorem \ref{thm:hardy} below) is an $L^p$ Hardy inequality on $M$, obtained (roughly speaking) by ``gluing'' together local Poincar\'e inequalities thanks to a suitable covering. Our approach also uses a localized version of the Calder\'on-Zygmund decomposition in Sobolev spaces as in \cite{AC}, already encountered in \cite{DR} (see Lemma \ref{lem:CZ} below).\par

\noindent The structure of the paper is as follows. Hardy inequalities are proved in Section \ref{sec:Hardy}. We then turn to the proof of Theorem \ref{thm:main} in Section \ref{sec:proofmainalt}. An appendix is devoted to the clarification of some properties of the Calder\'on-Zygmund decomposition. 
%while Section \ref{sec:subgauss} is devoted to the proofs of the reverse inequalities for quasi-Riesz transforms in subgaussian cases (section \ref{subsec:revQR}) and their generalizations (section \ref{subsec:gen}).
%We present the two major ingredients of the proof of Theorem \ref{thm:main}, namely the covering of $M$ by suitable balls and the proof of the Hardy inequality in Section \ref{sec:prelim} and prove Theorem \ref{thm:main} in Section \ref{sec:proofmain}. 

\begin{Rem}
{\em 

Some time after the first version of this paper has appeared, A. Sikora kindly informed us that the proof of the main result (Theorem \ref{thm:main}) had a mistake. His student D. He, in his PhD thesis (see \cite[Theorem 1.2]{He}) was then able to prove the main result in the special case of manifolds with quadratic curvature decay (namely, Corollary \ref{cor:main}). In the present manuscript, we present a full revised proof for the general case under the assumption $\nu>2$, and in the case of manifolds with quadratic curvature decay for $\nu>1$ (with a different proof from D. He's, although in both cases the main inspiration is \cite{C} as well as the Hardy inequalities from the first version of the present paper).

}
\end{Rem}

\section{Hardy inequalities}\label{sec:Hardy}
%\subsection{The Gaussian case}
Before stating the Hardy inequalities required for the proof of Theorem \ref{thm:main} and for the convenience of the reader, we feel it is worthwile to write down a more self-contained proof of the $L^p$ Hardy inequality in the case where $M$ is a connected sum of the Euclidean spaces of dimension $\geq 2$. It is well-known that on $\R^n$ the following optimal Hardy inequality holds:

\begin{equation}\label{eq:HardyR}\tag{$H_{\R^n}$}
\left(\frac{n-p}{p}\right)^{p}\int_{\R^n}\frac{|f|^p}{r^p}\leq\int_{\R^n}|df|^p,\quad \forall f\in C^\infty(\R^n).
\end{equation}
Hence, the Hardy inequality on a connected sum of two Euclidean spaces follows from the following result, which we think is of interest by itself:

\begin{Pro}

Let $M$ and $N$ be two Riemannian manifolds, such that $M$ and $N$ are isometric at infinity: there exists $K_M\Subset M$, $K_N\Subset N$ compact sets such that $M\setminus K_M$ is isometric to $N\setminus K_N$. Let $p\in (1,\infty)$. Then, the Hardy inequality
\begin{equation}\label{eq:Hardy}\tag{H}
\int_M\left(\frac{|f|}{r+1}\right)^p\lesssim \int_M|df|^p,\quad \forall f\in C^\infty_0(M).
\end{equation}
holds on $M$, if and only if it holds on $N$.

\end{Pro}

\begin{Cor}

Let $M=\R^n\sharp \R^n$ be a connected sum of two Euclidean spaces of dimension $n\geq 2$. Let $1\leq p<n$. Then, $M$ satisfies the $L^p$ Hardy inequality \eqref{eq:Hardy}.

\end{Cor}

\begin{proof} (of the proposition)

Assume that \eqref{eq:Hardy} holds on $N$. We are going to show that it holds on $M$ as well. By assumption, there exists two relatively compact, open sets $U\subset M$, $V\subset N$ such that $M\setminus U$ is isometric to $N\setminus V$. Let $0\leq \chi\leq 1$ be a smooth, compactly supported function on $M$, which is equal to $1$ identically in an neighborhood of $U$. Let $K\Subset M$ be a compact set containing the support of $\chi$. Let us take $f\in C^\infty_0(M)$, and write $f=\chi f+(1-\chi)f$. The function $(1-\chi)f$ identifies naturally with a smooth, compactly supported function defined on $N\setminus V$, hence the Hardy inequality on $N$ yields:

$$\int_{N}\left(\frac{|(1-\chi)f|}{1+r}\right)^p \lesssim \int_{N}|d\left((1-\chi)f\right)|^p.$$
Since $d\left((1-\chi)f\right)=-(d\chi)f+(1-\chi)df$, upon using the elementary inequality $(a+b)^p\leq 2^{p-1}(a^p+b^p)$ one gets:

\begin{eqnarray*}
\int_{N}\left(\frac{|(1-\chi)f|}{1+r}\right)^p &\lesssim &\int_{N}|d\chi|^p |f|^p+\int_{N}|1-\chi|^p|df|^p\\
&\lesssim &\int_K|f|^p+\int_M|df|^p
\end{eqnarray*}
On the other hand, one clearly has

$$\int_M \left(\frac{|\chi f|}{1+r}\right)^p\lesssim \int_K|f|^p,$$
so that, finally, one arrives to

\begin{eqnarray*}
\int_M \left(\frac{|f|}{1+r}\right)^p&\leq & 2^{p-1}\int_M \left(\frac{|\chi f|}{1+r}\right)^p+2^{p-1}\int_{M}\frac{|(1-\chi)f|^p}{r^p}\\
&\lesssim & \int_K|f|^p+\int_M|df|^p
\end{eqnarray*}
Now, the assumed Hardy inequality on $N$ implies that $N$ is $p$-hyperbolic (see \cite[Prop. 2.2]{BD}), and since $M$ and $N$ are isometric at infinity, it follows that the ends of $M$ are $p$-hyperbolic, hence $M$ itself is $p$-hyperbolic. For details, see \cite[Section 2]{BD}. Therefore, there exists a constant $C_K$ such that for every $u\in C_0^\infty(M)$,

$$\int_K|u|^p\leq C_K\int_M|du|^p.$$
Combining this inequality with the previous one, one obtains that

$$\int_M \left(\frac{|f|}{1+r}\right)^p \lesssim \int_M|df|^p,$$
which is precisely the sought for Hardy inequality \eqref{eq:Hardy} on $M$.

\end{proof}

Let us now state a more general result on $L^p$ Hardy inequalities that essentially stems from the work of V. Minerbe \cite{Min}:

\begin{Thm}\label{thm:hardy}

Let $M$ be a complete Riemannian manifold satisfying \eqref{eq:DV}, $\mathrm{(RCE})$ and \eqref{eq:RDV} for an exponent $\nu>1$. Let $1\leq p<\nu$, and assume that \eqref{eq:PEp} holds. Then the $L^p$ Hardy inequality \eqref{eq:Hardy} holds on $M$.
\end{Thm}

\begin{proof}

Note first that the (RCE) assumption implies that every end of $M$ satisfies the (RCA) condition considered in \cite{Min}. Next, \eqref{eq:DV} and \eqref{eq:PEp} implies that the proof of \cite[Lemma 2.10]{Min}, which provides $L^p$ Poincar\'{e} inequalities for subset of annuli, applies {\em mutatis mutandis} in our context. Given (RCA) in each end of $M$, one can then construct a ``good covering'' of $M$ (in the sense of \cite[Definition 1.1]{Min}) for the pair of measure $(\frac{\mathrm{dvol}}{1+r^p},\mathrm{dvol})$ as in \cite[Section 2.3.1]{Min}, and a weighted graph associated to this covering. The $L^p$ Poincar\'{e} inequalities for subset of annuli then implies that the good covering satisfies {\em continuous $L^p$ Sobolev inequalities of order $\infty$}, in the sense of \cite[Definition 1.3]{Min}. In fact, an $L^p$ Sobolev inequalities of order $\infty$ is just another terminology for an $L^p$ Poincar\'{e} inequality. The proof of \cite[Theorem 2.23]{Min} shows that the weighted graph satisfies an {\em isoperimetric inequality}. According to \cite[Theorem 1.8]{Min}, the continuous Sobolev inequality for the covering, together with the isoperimetric inequality for the weighted graph, imply the global $L^p$ Hardy inequality \eqref{eq:Hardy}.

\end{proof}
\begin{question}
{\em 
Let $M$ be a complete Riemannian manifold satisfying \eqref{eq:DV}, \eqref{eq:UE}, \eqref{eq:RDV} for an exponent $\nu>1$, and \eqref{eq:PEp} for some $1\leq p<\nu$; does the $L^p$ Hardy inequality \eqref{eq:Hardy} hold for $1\le p<\nu$ ? In other words, can the assumption (RCE) be replaced by \eqref{eq:UE} in the statement of Theorem \ref{thm:hardy}? For $p=2$, it is proved in \cite[Theorem 1.2]{L} that, under \eqref{eq:DV}, \eqref{eq:RDV} for some $\nu>2$ and \eqref{eq:UE}, an $L^2$ Hardy inequality holds, however the proof does not extend easily to the case $p\neq 2$ unless one knows {\em a priori} that \eqref{eq:RRp} holds (which of course we do not want to assume in the present paper).
}
\end{question}

\section{Proof of the $L^p$ reverse inequality} \label{sec:proofmainalt}
To begin with, let us recall that, under the assumptions of Theorem \ref{thm:main}, there exists a covering of $M$ by admissible balls, as well as an associated partition of unity. The following statement can be found in \cite[Section 2.1]{DR}: 
%Let us recall the concept of remote and admissible balls from \cite{DR}:
%
% \begin{definition} \label{balls}
%Let $x\in M$ and $r>0$. 
%\begin{enumerate}
%\item The ball $B(x,r)$ is called \textit{remote} if $r\leq \frac{r(x)}{2}$,
%\item The ball $B(x,r)$ is called \textit{anchored} if $x=o$,
%\item The ball $B(x,r)$ is \textit{admissible} if and only if $B$ is remote or $B$ is anchored and $r(B)\le r_0$, where $r_0>0$ is a fixed large constant.
%\end{enumerate}
%\end{definition}
\begin{Pro} \label{prop:goodcovering}
There exists a covering $(B_\alpha)_{\alpha\in \N}$ of $M$ by balls and an associated smooth partition of unity $(\chi_\alpha)_{\alpha\in \N}$ such that:
\begin{enumerate}
\item for every $\alpha\in \N$, the balls $B_\alpha$, $4B_\alpha$, $2^{100}B_\alpha$ are admissible,

\item the covering is {\em locally finite}: there exists $N\in\N$ such that for every $\alpha\in\N$,

$$\mathrm{Card}\{\beta\in\N\,;\,B_\alpha\cap B_\beta\neq \emptyset\}\leq N,$$
\item for every $R>0$, the set 

$$\{\alpha\in\N\,;\,B_\alpha\cap B(o,R)\neq \emptyset\}$$
is finite, and the same is true for the coverings by the balls $4B_\alpha$ as well as $2^{100}B_\alpha$,
\item for all $\alpha\in \N$, $0\leq \chi_\alpha\leq 1$ and $\chi_\alpha$ has support in $B_\alpha$. Moreover, there exists a constant $C>0$ such that, for every $\alpha\in\N$, $||\nabla \chi_\alpha||_\infty\leq \frac{C}{r_\alpha}$, where $r_{\alpha}$ is the radius of $B_{\alpha}$,
\item for all $\alpha\ne 0$, 
\begin{equation} \label{eq:remote}
2^{-10}r(x_\alpha)\le r_\alpha\le 2^{-9}r(x_\alpha).
\end{equation}
\end{enumerate}

\end{Pro}
One can assume that $B_0=B(o,r_0)$, and up to enlarging the value of $r_0$ and discarding a finite number of balls intersecting $B_0$, one can also assume that each of the remaining balls $B$ of the covering is such that $14\times 2^{100}B$ is remote. In the sequel, we thus assume that the balls have been relabeled in such a way that $B_0=B(o,r_0)$ and for $\alpha\neq 0$, $14B_\alpha$ is remote. Another remark we can make is that the conclusions in Proposition \ref{prop:goodcovering} are partially redundant: namely, in points (1) and (4), it is enough to state the results for the larger balls $2^{100}B_\alpha$. However, in order to make it clearer for the reader when later we will refer to points (1) and (4) of the proposition for the balls $B_\alpha$ and $4B_\alpha$, we choose to keep the above slightly redundant version. 

\medskip

\begin{Rem}\label{rem:Poinc_7B}
{\em 

Let us mention that the inequality \eqref{eq:remote} in Proposition \ref{prop:goodcovering} will play an important role in the last part of the proof of Theorem \ref{thm:main} in which the Hardy inequality will be utilized, see Lemma \ref{lem:Uterm} below. The assumptions of Theorem \ref{thm:main} imply that, for all $\alpha\in \N$, all balls inside $14B_\alpha$ support the $L^q$-Poincar\'e inequality; in particular, if $\tilde{B}\subset 2B_\alpha$, then $7\tilde{B}$ supports the $L^q$-Poincar\'e inequality. 
}
\end{Rem}
%Associated to this covering, there is a smooth partition of unity $(\chi_\alpha)_{\alpha\in\N}$ such that $0\leq \chi_\alpha\leq 1$, $\chi_\alpha$ has support in $B_\alpha$, and 
%
%$$||\nabla\chi_\alpha||_\infty\lesssim \frac{1}{r_\alpha}.$$

The idea for the proof of Theorem \ref{thm:main} is as follows: first, decompose $f$ into 

$$f=\sum_{\alpha\in \N} \chi_\alpha f=:\sum_{\alpha\in \N} f_\alpha.$$
Roughly speaking, we are going to estimate separately the ``diagonal terms'' $||\Delta^{1/2} f_\alpha||_{L^p(4B_\alpha)}$ and the ``off-diagonal'' terms $||\Delta^{1/2} f_\alpha||_{L^p(M\setminus 4B_\alpha)}$ for all $\alpha\in \N$ (the actual strategy is in fact a little bit more complicated as the reader will see in the sequel).

\noindent In the sequel of the argument, we use the following integral representation of $\Delta^{1/2}$:
$$
\Delta^{1/2}=c\int_0^{+\infty} \Delta e^{-t\Delta}\frac{dt}{\sqrt{t}},
$$
where $c>0$ is an unimportant constant. As in \cite[Section 1.2]{AC}, it is enough to prove the required estimates for $\int_{\varepsilon}^R \Delta e^{-t\Delta}\frac{dt}{\sqrt{t}}$ for $0<\varepsilon<R<\infty$, with constants independent of $\varepsilon,R$. In what follows, we ignore this issue most of the time and write directly $\int_0^{+\infty}$. For $\alpha\in\N$, let

$$T_\alpha =\int_0^{r_\alpha^2} \Delta e^{-t\Delta}\chi_\alpha \frac{dt}{\sqrt{t}},\quad \tilde{T}_\alpha =\int_0^{r_\alpha^2} \Delta e^{-t\Delta}\frac{dt}{\sqrt{t}}$$
and

$$U_\alpha=\int_{r_\alpha^2}^\infty \Delta e^{-t\Delta}\chi_\alpha \frac{dt}{\sqrt{t}},\quad \tilde{U}_\alpha=\int_{r_\alpha^2}^\infty \Delta e^{-t\Delta}\frac{dt}{\sqrt{t}}$$
and define two operators

$$U=\sum_{\alpha\in\N} U_\alpha,\quad T=\sum_{\alpha\in\N} T_\alpha ,$$
so that 

\begin{equation}\label{eq:T+U}
\Delta^{1/2}=c(T+U)
\end{equation}
Note that if $f$ has compact support, it follows from the properties of the covering $(B_\alpha)_\alpha$ that the sums in $Uf$ and $Tf$ are actually finite, hence $Uf$ and $Tf$ are well-defined. We further decompose $T$ and $U$ into ``diagonal'' and ``off-diagonal'' parts:

$$U_{\mathrm{diag}}=\sum_{\alpha\in\N} \mathbf{1}_{4B_\alpha}U_\alpha,\quad U_{\mathrm{off}}=\sum_{\alpha\in\N}\mathbf{1}_{M\setminus 4B_\alpha}U_\alpha,$$

$$T_{\mathrm{diag}}=\sum_{\alpha\in\N} \mathbf{1}_{4B_\alpha}T_\alpha,\quad T_{\mathrm{off}}=\sum_{\alpha\in\N}\mathbf{1}_{M\setminus 4B_\alpha}T_\alpha.$$
We therefore have, by definition:

$$T=T_{\mathrm{diag}}+T_{\mathrm{off}},\quad U=U_{\mathrm{diag}}+U_{\mathrm{off}}$$

\subsection{Estimate for the operator $U$}
We first explain how to deal with the operator $U$. For $\Omega\subset M$ an open set, let us denote by $\Lip_c(\Omega)$ the set of Lipschitz continuous functions with compact support in $\Omega$.

\begin{Lem}\label{lem:Uterm}

Let $p\in (1,+\infty)$. Then, for all $f\in C_c^\infty(M)$, and all $\alpha\in\N$, $g\in \Lip_c(2B_\alpha)$,

$$||Uf||_p\leq C\left\|\frac{f}{1+r}\right\|_p,\quad ||\tilde{U}_\alpha g||_p\leq \frac{C}{r_\alpha}||g||_p$$
with a constant $C$ independant of $\alpha$. Moreover, if $1<p<\nu$, then

$$||Uf||_p\lesssim ||\nabla f||_p,\quad ||\tilde{U}_\alpha g||_p\lesssim ||\nabla g||_p.$$

\end{Lem}

\begin{Cor}\label{cor:TbddL2}

Assume that $\nu>2$; then, for all $\alpha\in\N$, all $f\in C_c^\infty(M)$ and all $g\in Lip_c(2B_\alpha)$,

$$||Tf||_2\lesssim ||\nabla f||_2, \quad ||\tilde{T}_\alpha g||_2\lesssim ||\nabla g||_2.$$

\end{Cor}

\begin{Rem}
{\em 
If $1<\nu \leq 2$ but ($RR_q$) is true for some $q<\nu$, using Lemma \ref{lem:Uterm} and the fact that $\Delta^{\frac 12}=c(T+U)$, one still concludes that $||Tf||_q\lesssim ||\nabla f||_q$ and $||\tilde{T}_\alpha g||_q\lesssim ||\nabla g||_q$ for $g\in \Lip_c(2B_\alpha)$.  
}
\end{Rem}

\begin{proof}[Proof of Corollary \ref{cor:TbddL2}]

Thanks to Green's formula, for every $f\in C_0^\infty(M)$, there holds:

$$||\Delta^{1/2}f||_2=||\nabla f||_2$$
The first conclusion now follows from \eqref{eq:T+U} and Lemma \ref{lem:Uterm}. For the second one, note that 

$$\Delta^{1/2}g=c(\tilde{T}_\alpha +\tilde{U}_\alpha)g,$$
and apply Lemma \ref{lem:Uterm} again. 

\end{proof}

\begin{proof}[Proof of Lemma \ref{lem:Uterm}]

We follow a strategy introduced in \cite[Section 1.2]{AC}. Recall that by definition,

$$Uf=\sum_{\alpha\in\N}U_\alpha f,$$
where $f$ is an arbitrary function in $C_c^\infty(M)$. Let $g\in C_0^\infty(M)$. Note that since $f\in C_0^\infty(M)$, the above sum has only a finite number of nonzero terms, so we can without problem invert $\sum_{\alpha\in\N}$ with various $\int$ symbols. We estimate:

\begin{eqnarray*}
\left|\int_M (Uf)g\,dx \right| &=& \left| \int_0^\infty \sum_\alpha \langle t\Delta e^{-t\Delta} \left(\frac{1}{\sqrt{t}}\mathbf{1}_{[r^2_\alpha,\infty)}(t)f_\alpha\right),g\rangle \frac{dt}{t}\right|\\
&=& \left| \int_0^\infty \sum_\alpha \langle \left(\frac{1}{\sqrt{ t}}\mathbf{1}_{[r^2_\alpha,\infty)}(t)f_\alpha\right),t\Delta e^{-t\Delta} g \rangle \frac{dt}{t}\right|\\
& = & \left\vert \int_0^\infty  \langle \sum_\alpha \left(\frac{1}{\sqrt{t}}\mathbf{1}_{[r^2_\alpha,\infty)}(t)f_\alpha \right),t\Delta e^{-t\Delta} g \rangle \frac{dt}{t}\right\vert \\
& =:  & \left\vert  \int_0^\infty \langle F_t ,G_t \rangle \frac{dt}{t} \right\vert  \\
&\leq & \left\| \left(\int_0^\infty F_{t}^2\,\frac{dt}{t}\right)^{1/2}\right\|_p\left\|\left(\int_0^\infty  G_{t}^2\,\frac{dt}{t}\right)^{1/2}\right\|_{p'}\\
&\lesssim &  \left\| \left(\int_0^\infty F_{t}^2\,\frac{dt}{t}\right)^{1/2}\right\|_p\cdot||g||_{p'},
\end{eqnarray*}
where we have used Littlewood-Paley-Stein estimates in the last line, and we have let

$$F_t :=\sum_\alpha \frac{1}{\sqrt{t}}\mathbf{1}_{[r^2_\alpha,\infty)}(t)f_\alpha$$
and

$$G_t :=t\Delta e^{-t\Delta}g.$$
Thus, dividing by $||g||_{p'}$ and taking the supremum over $g\in C_0^\infty(M)\setminus\{0\}$, we obtain:

$$||Uf||_p\lesssim \left\| \left(\int_0^\infty |F_t|^2\frac{dt}{t}\right)^{1/2}\right\|_p.$$
By the finite intersection property for the balls $B_\alpha$,
\[
\left\vert F_t(x)\right\vert^2 \le N\sum_{\alpha}\frac{\left\vert f_\alpha(x)\right\vert^2}{t} \mathbf{1}_{[r_\alpha^2,+\infty)}(t).
\]
As a consequence,
\begin{eqnarray*}
\int_0^\infty |F_t(x)|^2\frac{dt}{t}& \le & N\sum_{\alpha} \int_0^{+\infty} \frac{\left\vert f_\alpha(x)\right\vert^2}{t} \mathbf{1}_{[r_\alpha^2,+\infty)}(t)\frac{dt}t\\
& \le & N\sum_{\alpha} \frac{\left\vert f_\alpha(x)\right\vert^2}{r_\alpha^2}\\
 & \leq & \max\left(N,N^{2-\frac 2p}\right) \left(\sum_{\alpha} \frac{\left\vert f_\alpha(x)\right\vert^p}{r_\alpha^p}\right)^{2/p},
\end{eqnarray*}
where, for $p<2$, the inequality $||a||_{\ell_q(\N)}\leq ||a||_{\ell_1(\N)}$ has been used with $q=2/p>1$ and $a_\alpha=\frac{\left\vert f_\alpha(x)\right\vert^p}{r_\alpha^p}$, while, for $p\ge 2$, one has used the inequality
\[
\sum_{\alpha} c_\alpha\le N^{1-\frac 2p}\left(\sum_{\alpha} c_\alpha^{\frac p2}\right)^{\frac 2p},\quad c_\alpha=\frac{|f_\alpha(x)|^2}{r_\alpha^2}
\]
whenever there at at most $N$ non zero terms in the sums under consideration.
This entails
\begin{eqnarray*}
\left\Vert \left(\int_0^\infty |F_t|^2\,\frac{dt}{t}\right)^{\frac 12}\right\Vert_{L^p}^p 
& \le & \max\left(N^{\frac p2},N^{p-1}\right)  \int_{M} \sum_{\alpha} \frac{\left\vert f_\alpha(x)\right\vert^p}{r_\alpha^p}d\mu(x)\\
& = & \max\left(N^{\frac p2},N^{p-1}\right) \sum_{\alpha}\frac 1{r_\alpha^p}\int_{B_\alpha} \left\vert f_\alpha(x)\right\vert^pd\mu(x)\\
& \underset{\eqref{eq:remote}}{\lesssim} & \max\left(N^{\frac p2},N^{p-1}\right)\sum_{\alpha}\int_{B_\alpha} \frac{\left\vert f_\alpha(x)\right\vert^p}{(1+r(x))^p}d\mu(x)\\
& \le & C\int_M \frac{\left\vert f(x)\right\vert^p}{(1+r(x))^p}d\mu(x),
\end{eqnarray*}
which finally implies

$$||Uf||_p\lesssim  \left\Vert \frac f{1+r}\right\Vert_p$$
and concludes the proof of the first inequality. The proof of the inequality

$$||\tilde{U}_\alpha g||_p\lesssim \left\Vert \frac g{1+r}\right\Vert_p,\quad \forall g\in \Lip_c(2B_\alpha)$$
is nearly identical. Finally, the last two inequalities,

$$||Uf||_p\lesssim ||\nabla f||_p,\quad ||\tilde{U}_\alpha g||_p\lesssim ||\nabla g||_p$$
are consequence of the first two and of the Hardy inequality from Theorem \ref{thm:hardy}.

\end{proof}
In conclusion, in order to show the main result of the paper, one only needs to analyse the operator $T$, and more specifically to prove:

\begin{equation}\label{eq:Tterm}
||Tf||_p\lesssim ||\nabla f||_p
\end{equation}
In the case of manifolds with quadratic curvature decay, this can be done directly by using the techniques introduced in the paper \cite{C}, and especially the gradient heat kernel estimates which hold in this context:

\begin{Pro}\label{pro:T_quadratic}

Under the assumptions of Corollary \ref{cor:main}, one has 

$$||Tf||_p\lesssim ||\nabla f||_p$$
for every $1<p<\nu$.

\end{Pro}

\begin{proof}

Let $g\in L^q$, $q=p'$; then, integrating by parts,

\begin{eqnarray*}
\langle Tf,g\rangle &=& \sum_\alpha \int_0^{r_\alpha^2}\langle \nabla f_\alpha,\sqrt{t}\nabla e^{-t\Delta}g\rangle \,\frac{dt}{t} \\
&=& \sum_{\alpha}\langle \nabla f_\alpha,\int_0^{r_\alpha^2}\sqrt{t}\nabla e^{-t\Delta}g \,\frac{dt}{t}\rangle\\
&=& \sum_{\alpha}\langle \nabla f_\alpha,\int_0^{r_\alpha^2}\mathbf{1}_{B_\alpha}\sqrt{t}\nabla e^{-t\Delta}g \,\frac{dt}{t}\rangle\\
&=&\sum_{\alpha}\langle \nabla f_\alpha,\int_0^{r_\alpha^2}\mathbf{1}_{B_\alpha}\sqrt{t}\nabla e^{-t\Delta}\mathbf{1}_{4B_\alpha}g \,\frac{dt}{t}\rangle\\
&&+ \sum_{\alpha}\langle \nabla f_\alpha,\int_0^{r_\alpha^2}\mathbf{1}_{B_\alpha}\sqrt{t}\nabla e^{-t\Delta}\mathbf{1}_{M\setminus 4B_\alpha}g \,\frac{dt}{t}\rangle\\
&=& I+J,
\end{eqnarray*}
where to pass from the second to the third line, we have used the fact that $f_\alpha$ has support in $B_\alpha$. In order to estimate $J$, we note that, thanks to the covering $(B_\alpha)_\alpha$ having finite intersection, to point (5) of Prop. \ref{prop:goodcovering}, and to the estimate $|\nabla \chi_\alpha|\leq \frac{C}{r_\alpha}$ for every $\alpha$, there holds:

\begin{eqnarray*}
|J|&\leq & N \int_{d(x,y)\geq 2^{-10}r(x)}\int_0^{r^2(x)} \left(|\nabla f(x)|+\frac{C|f(x)|}{r(x)}\right)\\
&&\times \sqrt{t}|\nabla_xp_t(x,y)|\cdot|g(y)|\,dxdy\frac{dt}{t},
\end{eqnarray*}
where $p_t(x,y)$ is the heat kernel. (Note that we have used the fact that $r_\alpha\leq r(x)$ if $x\in B_\alpha$, which is a consequence of point (5) of Prop. \ref{prop:goodcovering}). 
%\todo{I do not see where the factor \(2^{10}\) comes from. It seems that
%\[
%r_\alpha\le 2^{-9}r(x_\alpha)\le 2^{-9}(r(x)+r_\alpha),
%\]
%which entails
%\[
%r_\alpha\le \frac{2^{-9}}{1-2^{-9}}r(x)\le 2^{-8}r(x).
%\]
%This is probably not so important, but the factor \(2^{10}\) is quite strange. We could modify the sequel accordingly. \Mg OK, je suis d'accord. Vu tes inegalites, on peut juste dire que $r_\alpha\leq r(x)$, ce qui donne des quantites plus plaisantes dans les integrales.\Bk
%}
Thus, if one denotes

$$K(x,y)=\mathbf{1}_{d(x,y)\geq 2^{-10}r(x)} \int_0^{r^2(x)} \sqrt{t}|\nabla_xp_t(x,y)|\,\frac{dt}{t},$$
one has by H\"older's inequality

$$|J|\lesssim \left(||\nabla f||_p+\left|\left|\frac{f}{1+r}\right|\right|_p\right) \times ||Kg||_q$$
Thus, one is lead to prove that

\begin{equation}\label{eq:Kg}
||Kg||_q\lesssim ||g||_q
\end{equation}
According to the parabolic Li-Yau estimates (see for instance \cite[Section 3.3]{C}), under our hypothesis of quadratic decay of the Ricci curvature, there holds:

$$\sqrt{t}|\nabla_xp_t(x,y)|\lesssim \frac{1}{V(x,\sqrt{t})}e^{-c\frac{d^2(x,y)}{t}},\quad t\leq  r^2(x)$$
Using this, it is thus enough to prove that the operator $\tilde{K}$ with kernel

$$\tilde{K}(x,y)=\mathbf{1}_{d(x,y)\geq 2^{-10}r(x)}\int_0^{r^2(x)} \frac{1}{V(x,\sqrt{t})}e^{-c\frac{d^2(x,y)}{t}}\,\frac{dt}{t},$$
is bounded on $L^q$. Using doubling, one easily sees that

$$\tilde{K}(x,y)\lesssim \mathbf{1}_{d(x,y)\geq 2^{-10}r(x)}\frac{1}{V(o,r(x))}.$$
Further decomposing the RHS into two parts:

$$\mathbf{1}_{\Omega_0}(x,y)\frac{1}{V(o,r(x))}+\mathbf{1}_{\Omega_1}(x,y)\frac{1}{V(o,r(x))}:=K_0(x,y)+K_1(x,y),$$
with 
$$\Omega_0=\{r(x)\geq 2^{10}r(y)\},$$
and

$$\Omega_1=\{r(y)\geq 2^{10}r(x)\},$$
one concludes as in the proof of \cite[Prop. 4.1]{C} that $K_0$ is bounded on $L^q$, whatever the value of $q\in (1,+\infty)$ is. Next, since $K_1(x,y)=K_0(y,x)$, one has $K_1=(K_0)^*$, hence $K_1$ is bounded on $L^q$, too. Finally, we conclude that 

$$|J|\lesssim \left(||f||_p+\left|\left|\frac{f}{1+r}\right|\right|_p\right)\,||g||_q,$$
and thanks to Hardy's inequality for $1<p<\nu$, we obtain

$$|J|\lesssim ||\nabla f||_p||g||_q.$$
Now we turn to estimating $I$. For $B$ a remote ball of radius $r$, writing

$$\int_0^{r^2}\sqrt{t}\nabla e^{-t\Delta}\mathbf{1}_{4B}\,\frac{dt}{t}=\int_0^{(\tilde{r}^2/16)}\sqrt{t}\nabla e^{-t\Delta}\mathbf{1}_{\tilde{B}} \,\frac{dt}{t},$$
with $\tilde{B}=4B$ and $\tilde{r}=4r$, implies by a small variation of \cite[Lemma 4.3]{C} that the operator

$$\tilde{R}_B=\mathbf{1}_{4B} \int_0^{r^2}\sqrt{t}\nabla e^{-t\Delta}\mathbf{1}_{4B}\,\frac{dt}{t}$$
is bounded on $L^q$ for any $q\in (1,+\infty)$ with a norm which can be bounded independently of the remote ball $B$. Consequently, the same holds for $R_B=\mathbf{1}_{B}\tilde{R}_{B}$. We thus conclude, using the fact that the covering has finite intersection:

\begin{eqnarray*}
|I| & \lesssim & \langle \left(|\nabla f|+\frac{|f|}{1+r}\right), \sum_\alpha R_{B_\alpha}g\rangle\\
&\lesssim & \left(||\nabla f||_p+\left|\left|\frac{f}{1+r}\right|\right|_p\right) \left|\left|\sum_\alpha R_{B_\alpha}g\right|\right|_q.\\
\end{eqnarray*}
We further estimate, for $x\in M$:

\begin{eqnarray*}
\left|\sum_\alpha R_{B_\alpha}g(x)\right| & \leq & \sum_{x\in B_\alpha} \left|R_{B_\alpha}g(x)\right|\Bk \\
& \leq & N^{1/p}\left(\sum_{x\in B_\alpha} |R_{B_\alpha}g(x)|^q\right)^{1/q}, 
\end{eqnarray*}
(since the sum has at most $N$ terms), hence

$$\left|\left|\sum_\alpha R_{B_\alpha}g\right|\right|_q^q \leq N^{q-1} \sum_\alpha ||R_{B_\alpha}g||_q^q.$$
Writing

$$R_{B_\alpha}g=R_{B_\alpha}\mathbf{1}_{4B_\alpha}g$$
and using that $||R_{B_\alpha}||\lesssim 1$, we conclude that the RHS is bounded above by $\sum_\alpha ||\mathbf{1}_{4B_\alpha}g||_q^q\lesssim ||g||_q^q$, since the balls $(4B_\alpha)_\alpha$ have finite intersection. Thus, 

$$|I|\lesssim \left(||\nabla f||_p+\left|\left|\frac{f}{1+r}\right|\right|_p\right) \cdot ||g||_q,$$
and thanks to Hardy's inequality, we obtain for $1<p<\nu$ that

$$|I|\lesssim ||\nabla f||_p ||g||_q.$$

\end{proof}
We now turn to the more general case of manifolds satisfying Poincar\'e inequalities in the ends. The proof will be split into two parts: in the first part, we prove the estimate for the diagonal part $T_{\mathrm{diag}}$; in the second part, we prove it for $T_{\mathrm{off}}$.

\subsection{Estimates of the diagonal terms}
We now explain how to deal with the ``diagonal'' term $T_{\mathrm{diag}}$, using ideas from \cite{AC}. The main tool is a precise localized Calder\'on-Zygmund decomposition for gradients of functions, which is a variation on \cite[Prop. 1.1]{AC}. Define the (uncentered) maximal function $\mathscr{M}$ by

$$\mathscr{M}u(x)=\sup_{x\ni B}\frac{1}{\mu(B)}\int_B \left\vert u\right\vert \,d\mu,$$
for all functions $u\in L^1_{loc}(M)$ and all $x\in M$. The required Calder\'on-Zygmund decomposition is as follows:

\begin{lemma}\label{lem:CZ}
Let $B$ be a ball in $M$, and $u\in \Lip_c(B)$. Let $1\leq q<\infty$, and assume that, for all balls $\widetilde{B}\subset 2B$, the Poincar\'e inequality with exponent $q$ holds in $7\tilde{B}$. Then, there exists a constant $C>0$ depending only on the doubling constant and the Poincar\'e constant, with the following property: for all $\lambda>\displaystyle\left(\frac{C\left\Vert \nabla u\right\Vert_q^q}{\mu(B)}\right)^{\frac 1q}$, let

$$\Omega:=\{x\in M\,;\,\mathscr{M}(|\nabla u|^q)(x)>\lambda^q\}.$$
Then, $\Omega\subset 2B$, and there exists a denumerable collection of balls $(B_i)_{i\ge 1}\subset \Omega\subset 2B$ covering $\Omega$, a denumerable collection of $C^1$ functions $(b_i)_{i\ge 1}$ and a (globally) Lipschitz function $g$ such that, denoting by $r_i$ the radius if $B_i$:
\begin{enumerate}
\item $\displaystyle  u=g+\sum_{i\ge 1}b_i$,
\item the support of $g$ is included in $2B$, and $|\nabla g(x)|\leq C \lambda$, for a.e. $x\in M$. Moreover, there exists a bounded vector field $H\in L^\infty(TM)$ vanishing outside $\Omega$, such that

\begin{equation}\label{eq:nabla_g}
\nabla g =\nabla u\cdot \mathbf{1}_{M\setminus \Omega}+H\,\,\mbox{a.e.},\quad ||H||_\infty\leq C \lambda,
\end{equation}
\item the support of $b_i$ is included in $B_i$,

$$\displaystyle \int_{B_i} |b_i|^qd\mu\leq C r_i^q\int_{B_i}|\nabla u|^q\,d\mu,$$
and

$$\displaystyle \int_{B_i} |\nabla b_i|^qd\mu\leq C r_i^q\int_{B_i}|\nabla u|^q\,d\mu \leq C^2\lambda^q V(B_i).$$
\item $\displaystyle \sum_{i\ge 1} V(B_i)\leq C \lambda^{-q}\int |\nabla u|^qd\mu$,
\item there is a finite upper bound $N$ for the number of balls $B_i$ that have a non-empty intersection,
\item if $B_i\cap B_j\neq \emptyset$, then 

$$\frac{1}{3}r_j\leq r_i\leq 3r_j,$$

\item for every $i\in\N$, 

$$3B_i\cap (M\setminus \Omega)\neq \emptyset.$$

\end{enumerate}
\end{lemma}
The construction of the covering and of the functions $(b_i)_{i\ge 1}$ has been explained in details in \cite[Appendix B]{DR}; property 3 is an easy application of the Poincar\'e inequality which holds for every ball $B_i$. It turns out that property (2) is subtle, in fact the proof of Proposition 1.1 in \cite{AC} has a gap, which has subsequently been addressed in the unpublished note \cite{A}. For the sake of clarification of this point, we provide a proof of points 2-4 from Lemma \ref{lem:CZ} in the Appendix.\par

\begin{Rem}\label{rem:CVL1}
{\em 
As the proof in Appendix will demonstrate, by construction, the series $\sum_{i\ge 1} b_i$ converges in $L^1$ to a Lipschitz function $b$ and $\sum_{i\geq 1}\nabla b_i$ converges in $L^1$ to $\nabla b$.
}
%the two series $\sum_{i\geq 1}\nabla b_i$ and $\sum_{i\geq 1}b_i$ converge in $L^1(M)$, and equal respectively to $\nabla b$ and $b$ (which thus belong to $L^1$, too). 

\end{Rem}
It is important to notice, as a consequence of the properties of the Calder\'on-Zygmund decomposition (Lemma \ref{lem:CZ}):

\begin{Pro}\label{pro:Lq_CV_nabla_bi}

The series $\sum_{i\geq 1}\nabla b_i$ converges to $\nabla b$ in $L^q(M)$; moreover, the following estimate holds: 

\begin{equation}\label{eq:est_nabla_b}
||\nabla b||_q\leq \sum_i||\nabla b_i||_q\leq C' \min\left(\lambda \mu(B)^{1/q},\lambda\mu(\Omega)^{1/q},||\nabla u||_q\right)
\end{equation}
for some constant $C'>0$ depending only on the doubling constant.

\end{Pro}

\begin{proof}

First, by point (3) in Lemma \ref{lem:CZ}, one has
\[
\sum_i \left\Vert \nabla b_i\right\Vert_q^q\le C\lambda^q \sum_i \mu(B_i)\le C\left\Vert \nabla u\right\Vert_q^q.
\]
By the finite overlap property of the balls $B_i$ and H\"older's inequality,
\[
\left(\sum_i \left\vert \nabla b_i(x)\right\vert\right)^q \le N^{q-1}\sum_i \left\vert \nabla b_i(x)\right\vert^q 
\]
for all \(x\in M\). It follows that, for \(\ell,p\ge 1\),
\[
\left\Vert \sum_{i=l}^{\ell+p} \nabla b_i\right\Vert_q^q  \le  N^{q-1}\sum_{i=l}^{\ell+p} \left\Vert \nabla b_i\right\Vert_q^q,
\]
which entails that the sequence \(\left(\sum_{i=1}^\ell \nabla b_i\right)_{\ell\ge 1}\) is a Cauchy sequence in \(L^q\). Since \(\sum_i \nabla b_i\) converges to \(\nabla b\) in \(L^1\), the convergence also holds in \(L^q\). This, in turn, justifies the inequality 
\begin{equation}\label{eq:nabla_b}
||\nabla b||_q\le \sum_{i\geq 1}||\nabla b_i||_q\lesssim \lambda \left(\sum_{i\ge 1}\mu(B_i)\right)^{1/q}\lesssim  ||\nabla u||_q.
\end{equation}
It remains to prove the estimate \eqref{eq:est_nabla_b}. By the above inequalities,

$$||\nabla b||_q\le \sum_{i\geq 1}||\nabla b_i||_q\lesssim \lambda \left(\sum_{i\ge 1}\mu(B_i)\right)^{1/q}.$$
%\todo{\Bl I guess we could justify the convergence of \(\sum_i \nabla b_i\) in \(L^q\) by the following argument. First,
%\[
%\sum_i \left\Vert \nabla b_i\right\Vert_q^q\le C\lambda^q \sum_i \mu(B_i)\le C\left\Vert \nabla u\right\Vert_q^q.
%\]
%By the finite overlap property and H\"older,
%\[
%\left(\sum_i \left\vert \nabla b_i(x)\right\vert\right)^q \le N^{q-1}\sum_i \left\vert \nabla %b_i(x)\right\vert^q 
%\]
%for all \(x\in M\). It follows that, for \(l,p\ge 1\),
%\[
%\left\Vert \sum_{i=l}^{l+p} \nabla b_i\right\Vert_q^q  \le  N^{q-1}\sum_{i=l}^{l+p} \left\Vert \nabla b_i\right\Vert_q^q,
%\]
%which entails that the sequence \(\left(\sum_{i=1}^l \nabla b_i\right)_{l\ge 1}\) is a Cauchy sequence in \(L^q\). Since \(\sum_i \nabla b_i\) converges to \(\nabla b\) in \(L^1\), the convergence also holds in \(L^q\). This, in turn, justifies the inequality $$||\nabla b||_q\le \sum_{i\geq 1}||\nabla b_i||_q\lesssim \lambda \left(\sum_{i\ge 1}\mu(B_i)\right)^{1/q}.$$ \Bk \Mg OK, j'ai implemente ton argument dans la preuve\Bk}

By the finite intersection property and doubling

$$\sum_{i\ge 1}\mu(B_i)\leq N\mu(\cup_{i\geq 1}B_i)=N\mu(\Omega)\leq N\mu(2B)\lesssim N\mu(B),$$
hence

$$||\nabla b||_q\lesssim \lambda \min\left(\mu(B)^{1/q},\mu(\Omega)^{1/q}\right),$$
which together with \eqref{eq:nabla_b} concludes the proof of \eqref{eq:est_nabla_b}.

\end{proof}
\Bk

\bigskip

\noindent Let us now turn to the estimate of $\left\Vert T_{\mathrm{diag}}f\right\Vert_{L^p}$. Following partially \cite{AC}, we first prove:
\begin{lemma} \label{lem:muh}
Assume that $q\in (1,2]$. For all $\alpha\in \N$, all $f \in C^{\infty}_c(M)$ and all $\lambda>\sup_\beta \left(\displaystyle\frac{C\left\Vert \nabla f_\beta\right\Vert_q^q}{\mu(B_\beta)}\right)^{\frac 1q}$,
\begin{equation} \label{eq:muh} \mu(\{x\in 4B_\alpha;\ |T_{\mathrm{diag}}f(x)|>\lambda\})\lesssim \frac{1}{\lambda^q} \sum_\beta \int_{B_\beta} |\nabla f_\beta|^q\,d\mu,
\end{equation}
where the sup and the sum are taken over all $\beta$'s such that $4B_\alpha\cap 4B_\beta\neq \emptyset$.
\end{lemma}
\begin{proof}
Recall that

$$T_{\mathrm{diag}}=\sum_{\beta\in\N}\mathbf{1}_{4B_\beta}T_\beta,\quad T_\beta f=\tilde{T}_\beta f_\beta=\int_0^{r_\beta^2}\Delta e^{-t\Delta}f_\beta \frac{dt}{\sqrt{t}}.$$
As mentioned before, it is enough to prove the required estimates for 

$$\int_{\varepsilon}^{\min(r_\beta^2,R)} \Delta e^{-t\Delta}\chi_\beta\frac{dt}{\sqrt{t}}$$
for $0<\varepsilon<R<\infty,$ with constants independent of $\varepsilon,R$. In what follows, we ignore this issue and write directly $\int_0^{r_\beta^2}$. Note also that by the very construction of the covering $(B_\beta)_\beta$, the sum defining $T_{\mathrm{diag}}f(x)$ for any given $x\in M$ is actually finite, and even more, the number of non-zero terms is bounded from above by some $L\in\N$ independent from $x$. Since two intersecting balls $4B_\alpha$ and $4B_\beta$ have comparable radius with a quantitative estimate (as a consequence of point (5) in Proposition \ref{prop:goodcovering}), it follows that 

$$\mu(\{x\in 4B_\alpha;\ |T_{\mathrm{diag}}f(x)|>\lambda\})\leq L\sup_\beta \mu\left(\left\{x\in 2^{100}B_\beta;\ |\tilde{T}_\beta f_\beta(x)|>\frac{\lambda}{L}\right\}\right)$$
where the supremum is taken over all $\beta$'s such that $4B_\beta\cap4B_\alpha\neq \emptyset$. 

Since the balls $2^{100}B_\beta$ have finite intersection, up to changing the value of the constant $C$ and replacing it by $CL^q$, it is enough to prove the following estimate: for every $\lambda>\left(\displaystyle\frac{C\left\Vert \nabla f_\beta\right\Vert_q^q}{\mu(B_\beta)}\right)^{\frac 1q}$,

\begin{equation}\label{eq:mub}
\mu\left(\{x\in 2^{100}B_\beta;\ |\tilde{T}_\beta f_\beta(x)|>\lambda\}\right) \lesssim\frac{1}{\lambda^q}  \int_{B_\beta} |\nabla f_\beta|^q\,d\mu.
\end{equation}
We will need later a slightly stronger variant of this estimate, namely: for every $\varphi\in \Lip_c(2B_\beta)$ and 
every $\lambda>\left(\displaystyle\frac{C\left\Vert \nabla \varphi\right\Vert_q^q}{\mu(B_\beta)}\right)^{\frac 1q}$,

\begin{equation}\label{eq:weak_Ta}
\mu\left(\{x\in 2^{100}B_\beta;\ |\tilde{T}_\beta \varphi(x)|>\lambda\}\right) \lesssim \frac{1}{\lambda^q}  \int_{2B_\beta} |\nabla \varphi|^q\,d\mu.
\end{equation}
Let us now prove \eqref{eq:weak_Ta}. Given the assumption made on $\lambda$, Lemma \ref{lem:CZ}, applied to $B=2B_\beta$ (in what follows we remove the dependance in $\beta$ for the ease of notations and keep it only in $\tilde{T}_\beta$ as a reminder) and $u=\varphi$, provides a collection of balls $(B_i)_{i\geq 1}$ included in $2B$, a Lipschitz function $g$ and a collection of $C^1$ functions $(b_i)_{i\geq 1}$ sharing the properties listed in Lemma \ref{lem:CZ}. In particular,
$$
\varphi=g+\sum_i b_i=g+b. 
$$
Note that the fact that Lemma \ref{lem:CZ} is applicable to $B=2B_\beta$ follows from the fact that $14B_\beta$ is remote, see Remark \ref{rem:Poinc_7B}. Note also that, for all $i\ge 1$, $B_i\subset 2B$ which is remote, hence the balls $B_i$ satisfy the $L^q$ Poincar\'e inequality. \par
The meaning of $\tilde{T}_\beta g$ and $\tilde{T}_\beta b_i$ is analogous to the one given in \cite[Section 1.2]{AC} and relies on the pointwise Gaussian upper bounds \eqref{eq:UE} and \eqref{eq:dUE} for $p_t(x,y)$ and $\left\vert \frac{\partial p_t}{\partial t}(x,y)\right\vert$ respectively.\par
%\todo{This seems to be correct, since $f_\beta\in C^{\infty}_c(M)$ and $p_t(x,y)$ and $\left\vert \frac{\partial %p_t}{\partial t}(x,y)\right\vert$ satisfy Gaussian upper estimates.}
\noindent We first claim that
\begin{equation} \label{eq:mug}
\mu\left(\left\{x\in 2^{100}B;\ \left\vert \tilde{T}_\beta g(x)\right\vert>\frac{\lambda}3\right\}\right)\le \frac{C}{\lambda^q}\int_{2B} \left\vert \nabla \varphi(x)\right\vert^qd\mu(x).
\end{equation}
Indeed, thanks to Corollary \ref{cor:TbddL2} (this is where we need the assumption $\nu>2$),

\begin{eqnarray*}
\mu\left(\left\{x\in 2^{100}B;\ \left\vert \tilde{T}_\beta g(x)\right\vert>\frac{\lambda}3\right\}\right) & \le & \frac{9}{\lambda^2}\int_{2^{100}B} \left\vert \tilde{T}_\beta g(x)\right\vert^2d\mu(x)\\
 & \le & \frac{9}{\lambda^2}\int_{M} \left\vert \tilde{T}_\beta g(x)\right\vert^2d\mu(x)\\
& \le & \frac{9}{\lambda^2}\int_{M} \left\vert \nabla g(x)\right\vert^2d\mu(x)\\
& \lesssim & \frac{1}{\lambda^2} \lambda^{2-q} \int_{M} \left\vert \nabla g(x)\right\vert^qd\mu(x)\\
& \lesssim & \frac{1}{\lambda^q} \int_{M} \left\vert \nabla \varphi (x)\right\vert^qd\mu(x)\\
&=& \frac{1}{\lambda^q} \int_{2B} \left\vert \nabla \varphi (x)\right\vert^qd\mu(x).
\end{eqnarray*}
Note that the fourth line holds since \(q\le 2\). Note also that the line before the last one is a consequence of Proposition \ref{pro:Lq_CV_nabla_bi} and of the fact that \(\nabla g=\nabla\varphi-\nabla b\).
%\todo{(Baptiste) Pour passer de la 3e a la 4e ligne on utilise deja que $q\leq 2$}  v

%due to the fact that, since the balls $B_i$ have bounded overlap,
%\begin{eqnarray*}
%\int_{M} \left\vert \sum_{i} \nabla b_i(x)\right\vert^qd\mu(x)
%& \lesssim & \sum_{i} \int_M \left\vert\nabla b_i(x)\right\vert^qd\mu(x)\\
%& \lesssim & \lambda^q \sum_i V(B_i)\\
%& \lesssim & \int_M \left\vert \nabla \varphi (x)\right\vert^qd\mu(x),
%\end{eqnarray*}
%which implies in turn that
%\begin{eqnarray*}
%\left\Vert \nabla g\right\Vert_q & \le & \left\Vert \nabla \varphi \right\Vert_q+\left\Vert \sum_i \nabla b_i\right\Vert_q\\
%& \lesssim & \left\Vert \nabla \varphi\right\Vert_q.
%\end{eqnarray*}
Define $b:=\sum_i b_i$. We claim that
\begin{equation} \label{eq:tbi}
\tilde{T}_\beta b=\sum_i \tilde{T}_\beta b_i
\end{equation}
where the series converges in $L^1(M)$. Indeed, remember that we are in fact working with 

\begin{equation}\label{eq:TeR}
\tilde{T}_\beta^{\varepsilon,R}=\int_\varepsilon^{R\wedge r_\beta^2}\Delta e^{-t\Delta}\, \frac{dt}{\sqrt{t}}.%\underset{(2)}{=}\int_\varepsilon^{R\wedge r_\beta^2} d^*e^{-t\vec{\Delta}}d\,\frac{dt}{\sqrt{t}}\right),
\end{equation}
%where $d^*$ is the co-differential (formal adjoint of the differential $d$) and the Laplacian writes $\Delta=d^*d$; moreover, $\vec{\Delta}$ is the so-called Hodge Laplacian on differential $1$-forms. 
Clearly, \eqref{eq:TeR} and the bound

$$||\Delta e^{-t\Delta}||_{1\to 1}\leq C,\quad \forall t\in [\varepsilon,R],$$
(which follows from \eqref{eq:dUE})
%\todo{We could simply invoke \eqref{eq:dUE} or \eqref{eq:dUE2}. (Baptiste) OK, modifié}
imply that $\tilde{T}_{\varepsilon,R}$ is a bounded operator on $L^1$. Since $b=\sum_{i\ge 1}b_i$ in $L^1$, it follows that $$\tilde{T}_\beta^{\varepsilon,R}b=\tilde{T}_\beta^{\varepsilon,R}\left(\sum_{i\geq 1}b_i\right)=\sum_{i\geq 1}\tilde{T}_\beta^{\varepsilon,R}b_i,$$
proving \eqref{eq:tbi}.
\Bk
%\todo{Another question: we have to compare %$\tilde{T}_\beta\left(\sum_i b_i\right)$ and $\sum_i %\tilde{T}_\beta b_i$.}
To cope with the terms involving $\tilde{T}_\beta b_i$, decompose
\begin{eqnarray}
\tilde{T}_\beta b_i & = & c \int_0^{r_\beta^2} \Delta e^{-t\Delta}b_i\frac{dt}{\sqrt{t}} \nonumber\\
& = & c\int_0^{r_i^2\wedge r_\beta^2} \Delta e^{-t\Delta}b_i\frac{dt}{\sqrt{t}}+c\int_{r_i^2\wedge r_\beta^2}^{r_\beta^2} \Delta e^{-t\Delta} b_{i}\frac{dt}{\sqrt{t}} \nonumber\\
& =: & T_i b_i+U_i b_i \label{eq:split}.
\end{eqnarray}
Notice that, by arguments similar to the one for $\tilde{T}_\beta^{\varepsilon,R}$, one can establish that the series $\sum_i T_ib_i$ and $\sum_i U_ib_i$ both converge in $L^1(M)$. Therefore,
\[
\sum_i \tilde{T}_\beta b_i=\sum_i T_ib_i+\sum_i U_ib_i
\]
where all the series converge in $L^1(M)$.
\par
\noindent We therefore have to establish
\begin{equation} \label{eq:estimTi}
I:=\mu\left(\left\{x\in 2^{100}B;\ \left\vert \sum_{i} T_i b_i(x)\right\vert>\frac{\lambda}3\right\}\right)\lesssim \frac 1{\lambda^q} \left\Vert \nabla \varphi\right\Vert_q^q
\end{equation}
and
\begin{equation} \label{eq:estimUi}
J:=\mu\left(\left\{x\in 2^{100}B;\ \left\vert \sum_{i} U_i b_i(x)\right\vert>\frac{\lambda}3\right\}\right)\lesssim \frac 1{\lambda^q} \left\Vert \nabla \varphi\right\Vert_q^q.
\end{equation}
Let us first consider \eqref{eq:estimTi}. The quantity $I$ is easily estimated by
\begin{eqnarray*}
I & \le & \mu\left(\bigcup_i 2B_i\right)+\mu\left(\left\{x\in 2^{100}B\setminus \left(\bigcup_i2Bi\right);\ \left\vert \sum_{i} T_i b_i(x)\right\vert>\frac{\lambda}3\right\}\right)\\
& =: & I_1+I_2.
\end{eqnarray*}
First, \eqref{eq:DV} and the properties of the Calder\'on-Zygmund decomposition yield at once
\begin{eqnarray*}
I_1 & \le & \sum_i V(2B_i)\\
& \lesssim & \sum_i V(B_i)\\
& \lesssim & \lambda^{-q}\int_M \left\vert \nabla \varphi (x)\right\vert^qd\mu(x).
\end{eqnarray*}
As far as $I_2$ is concerned, one has
\begin{eqnarray}
I_2 & \le &  \frac{9}{\lambda^2} \int_{2^{100}B\setminus \left(\bigcup_i2B_i\right)} \left\vert \sum_{i} T_i b_i(x)\right\vert^2d\mu(x) \nonumber\\
& \le &  \frac{9}{\lambda^2} \int_{2^{100}B} \left\vert \sum_i u_i(x)\right\vert^2d\mu(x) \label{eq:sumui},
\end{eqnarray}
where
$$
u_i:={\bf 1}_{2^{100}B\setminus \left(\bigcup_i2B_i\right)} \left\vert T_i b_i\right\vert.
$$
To estimate the right-hand side in \eqref{eq:sumui}, we argue by duality. Pick up a fonction $v\in L^2(2^{100}B)$ with $\left\Vert v\right\Vert_2=1$ and decompose
\begin{eqnarray}
\left\vert \int_{2^{100}B} \sum_i u_i(x) v(x)d\mu(x)\right\vert & = & \left\vert  \sum_i \sum_{j\ge 1} \int_{C_j(B_i)} u_i(x) v(x)d\mu(x) \right\vert \nonumber\\
& =: & \left\vert \sum_i \sum_{j\ge 1} A_{ij}\right\vert \label{eq:aij},
\end{eqnarray}
where
$$
C_j(B):=2^{j+1}B\setminus 2^jB
$$
for all open balls $B\subset M$ and all $j\ge 1$. In order to estimate $A_{ij}$, we need a pointwise upper bound for $\left\vert \frac{\partial}{\partial t}e^{-t\Delta}b_i\right\vert$ in $C_j(B_i)$, $j\geq 1$. So, let $j\ge 1$ and $x\in C_j(B_i)$. Denote by $x_i$ the center of $B_i$, and notice that \eqref{eq:VD} and \eqref{eq:eqvol} imply, for all $z\in B_i$ and all $t\in (0,r_{i}^2)$,

\begin{eqnarray*}
\frac{V(x_i,\sqrt{t})}{V(z,\sqrt{t})} &=& \frac{V(x_i,\sqrt{t})}{V(x_i,r_i)}\cdot \frac{V(x_i,r_i)}{V(z,r_i)}\cdot \frac{V(z,r_i)}{V(z,\sqrt{t})}\\
&\lesssim & \left(\frac{r_i}{\sqrt{t}}\right)^D.
\end{eqnarray*}
Bearing in mind that $b_i$ has support in $B_i$, that, for all $x\in C_j(B_i)$ and all $z\in B_i$, one has
$$
d(x,z)\ge d(x,x_i)-d(z,x_i)\ge (2^j-1)r_i\ge \frac 12 2^jr_i
$$
(recall that $j\ge 1$) and using \eqref{eq:dUE2}, one obtains, for all $x\in C_j(B_i)$,

\begin{eqnarray*}
\left| \frac{\partial}{\partial t}e^{-t\Delta}b_i(x)\right| &\lesssim &  \frac{1}{t} \left(\frac{r_i}{\sqrt{t}}\right)^D  \frac{e^{-c\frac{4^jr_i^2}{t}}}{V(x_i,\sqrt{t})}\int_{B_i}|b_i(z)|\,d\mu(z)\\
& \lesssim & \frac{1}{t}\left(\frac{r_i}{\sqrt{t}}\right)^D  e^{-c\frac{4^jr_i^2}{t}} \frac{V(x_i,r_i)}{V(x_i,\sqrt{t})}\fint_{B_i}|b_i(z)|\,d\mu(z)\\
&\lesssim & \frac{1}{t} \left(\frac{r_i}{\sqrt{t}}\right)^{2D}  e^{-c\frac{4^jr_i^2}{t}}  \left(\fint_{B_i}|b_i(z)|^q\,d\mu(z)\right)^{1/q}\\\
& \lesssim & \frac {r_i}t \left(\frac{r_i^2}{t}\right)^{D}  e^{-c\frac{4^jr_i^2}{t}} \frac 1{\left(V(B_i)\right)^{1/q}}\left\Vert \nabla \varphi\right\Vert_{L^q(B_i)} \\
& \le & \frac {r_i}t \left(\frac{r_i^2}{t}\right)^{D}  e^{-c\frac{4^jr_i^2}{t}} \lambda,
\end{eqnarray*}
where, in the third line, we have used \eqref{eq:VD} and H\"older's inequality, while the fourth one follows from point (3) in Lemma \ref{lem:CZ} and the last one from the definition of $\Omega$ and point (7) in Lemma \ref{lem:CZ}. As a consequence,
using doubling again, one obtains
\begin{eqnarray*}
\left\Vert \Delta e^{-t\Delta}b_i\right\Vert_{L^2(C_j(B_i))} & \le & \mu(C_j(B_i))^{1/2} \left\Vert \Delta e^{-t\Delta}b_i\right\Vert_{L^{\infty}(C_j(B_i))}\\
& \lesssim & 2^{jD/2}V(B_i)^{1/2} \frac {r_i}t \left(\frac{r_i^2}{t}\right)^{D}  e^{-c\frac{4^jr_i^2}{t}} \lambda.
\end{eqnarray*}
From thus, recalling that $u_i={\bf 1}_{2^{100}B\setminus \left(\bigcup_i2B_i\right)} \left\vert T_i b_i\right\vert$ we infer that
\begin{eqnarray}
\left\Vert u_i\right\Vert_{L^2(C_j(B_i))} & \leq & \left\Vert T_i b_i\right\Vert_{L^2(C_j(B_i))}\nonumber \\
& \le & \int_0^{r_i^2\wedge r_\beta^2} \left\Vert \Delta e^{-t\Delta}b_i\right\Vert_{L^2(C_j(B_i))} \frac{dt}{\sqrt{t}} \nonumber\\
& \lesssim & 2^{jD/2}V(B_i)^{1/2} \lambda \int_0^{r_i^2} \frac {r_i}t \left(\frac{r_i^2}{t}\right)^{D}  e^{-c\frac{4^jr_i^2}{t}} \frac{dt}{\sqrt{t}} \nonumber\\
& \lesssim & 2^{jD/2}V(B_i)^{1/2} \lambda \int_0^1 \frac 1{u} \left(\frac{1}{u}\right)^{D} e^{-c\frac{4^j}{u}}  \frac{du}{\sqrt{u}}\nonumber \\
& \lesssim & 2^{jD/2}V(B_i)^{1/2} e^{-\frac c24^j}\lambda, \label{eq:estimtibi}
\end{eqnarray}
where, we used $r_i^2\wedge r_\beta^2\leq r_i^2$ in the third line and, in the fourth line, we made the change of variables $t=r_i^2u$. \par
%\todo{The inequality $r_i^2\wedge r_\beta^2\leq r_i^2$ is used in the third line, I guess ? \textcolor{red}{Oui}}
\medskip
\noindent On the other hand, for all $y\in B_i$, 
\begin{eqnarray*}
\left(\int_{C_j(B_i)} \left\vert v(z)\right\vert^2d\mu(z)\right)^{\frac 12} & \le & \left(\int_{2^{j+1}B_i} \left\vert v(z)\right\vert^2d\mu(z)\right)^{\frac 12} \\
& \lesssim & V(2^{j+1}B_i)^{1/2} \left(\mathscr{M}(\left\vert v\right\vert^2)(y)\right)^{1/2}\\
&\lesssim & 2^{jD/2}V(B_i)^{1/2}\left(\mathscr{M}(\left\vert v\right\vert^2)(y)\right)^{1/2},
\end{eqnarray*}
and averaging over all $y\in B_i$, we obtain

\begin{equation}\label{eq:estimvcj}
\left(\int_{C_j(B_i)} \left\vert v(z)\right\vert^2d\mu(z)\right)^{\frac 12}  \lesssim 2^{jD/2}V(B_i)^{1/2}\fint_{B_i} \left(\mathscr{M}(\left\vert v\right\vert^2)(y)\right)^{1/2} d\mu(y).
\end{equation}
Gathering \eqref{eq:estimtibi} and \eqref{eq:estimvcj} and using Cauchy-Schwarz, one therefore obtains, remembering that $j\geq 1$,
\begin{eqnarray*}
\left\vert A_{ij}\right\vert & \le& 2^{jD}V(B_i) e^{-\frac c24^j}\lambda\cdot  \fint_{B_i} \left(\mathscr{M}(\left\vert v\right\vert^2)(y)\right)^{1/2} d\mu(y)\\
& \lesssim &  \lambda e^{- c'4^j}\int_{B_i} \left(\mathscr{M}(\left\vert v\right\vert^2)(y)\right)^{1/2} d\mu(y).
\end{eqnarray*}
%\todo{Is \eqref{eq:DV} really used here ? \textcolor{red}{Non, corrige}}
Summing up over $i,j$ and recalling \eqref{eq:aij} and the fact that $||v||_2=1$, one deduces
\begin{eqnarray}
\left\vert \int_{2^{100}B} \sum_i u_i(x) v(x)d\mu(x)\right\vert & \lesssim & \lambda \int_{M} \sum_i {\bf 1}_{B_i}(y) \left(\mathscr{M}(\left\vert v\right\vert^2)(y)\right)^{1/2} d\mu(y) \nonumber\\
& \lesssim & N\lambda \int_{\bigcup_i B_i}  \left(\mathscr{M}(\left\vert v\right\vert^2)(y)\right)^{1/2} d\mu(y)\nonumber\\
& \lesssim & N\lambda \mu\left(\bigcup_i B_i\right)^{1/2} \left\Vert  \left(\mathscr{M}(\left\vert v\right\vert^2)\right)\right\Vert_{1,\infty}^{1/2}\nonumber \\
& \lesssim & N\lambda \mu\left(\bigcup_i B_i\right)^{1/2} \left\Vert  \left\vert v\right\vert^2\right\Vert_{1}^{1/2}\nonumber\\
& \lesssim & N\lambda \mu\left(\bigcup_i B_i\right)^{1/2} \label{eq:uialphav},
\end{eqnarray}
where the second line follows from the finite overlap property for the balls $B_i$ (recall that $N$ is given by Lemma \ref{lem:CZ}), the third one is due to the Kolmogorov inequality (\cite[Lemma 10, Section 7.7]{Meyer}) and the fourth one to the weak $(1,1)$ boundedness of $\mathscr{M}$. Finally, taking the supremum over all functions $v\in L^2(2^{100}B)$ such that $\left\Vert v\right\Vert_{L^2(2^{100}B)}=1$ and recalling \eqref{eq:sumui}, we conclude
\begin{eqnarray*}
I_2 & \le & \frac{9}{\lambda^2} \int_{2^{100}B} \left\vert \sum_i u_i(x)\right\vert^2d\mu(x)\\
& \lesssim  & \mu\left(\bigcup_i B_i\right)\\
& \lesssim & \frac 1{\lambda^q} \left\Vert \nabla \varphi\right\Vert_q^q.
\end{eqnarray*}
Thus, \eqref{eq:estimTi} is proved.\par

\bigskip

\noindent Let us now turn to the proof of \eqref{eq:estimUi}.  We follow ideas in \cite[Section 1.2]{AC}. For each $i$,

$$U_i b_i=\int_{r_i^2\wedge r_\beta^2}^{r_\beta^2} t\Delta e^{-t\Delta}\left(\frac{b_i}{\sqrt{t}}\right)\,\frac{dt}{t}=\int_0^\infty t\Delta e^{-t\Delta} b_{i,t}\,\frac{dt}{t},$$
with
\[
b_{i,t}:=\frac{b_i}{\sqrt{t}}\mathbf{1}_{[r_i^2\wedge r_\beta^2,r_\beta^2)}(t).
\]
As a consequence,
\[
\sum_i U_i b_i=\int_0^\infty t\Delta e^{-t\Delta}f_t\,\frac{dt}{t},
\]
with
\[
f_t:=\sum_i\frac{b_i}{\sqrt{t}}\mathbf{1}_{[r_i^2\wedge r_\beta^2,r_\beta^2)}(t).
\]
Let $g\in C_c^\infty(M)$ with $\left\Vert g\right\Vert_{L^{q^{\prime}}}=1$, $\frac 1q+\frac 1{q^{\prime}}=1$. Since $q^{\prime}\in (1,+\infty)$, Littlewood-Paley-Stein estimates (\cite[Chapter 4, Theorem 10]{topics}) yield
\begin{eqnarray*}
\left\vert \int_{M} \left(\sum_iU_i b_i\right)gd\mu\right\vert  & = & \left\vert  \int_0^\infty \langle t\Delta e^{-t\Delta} f_t,g\rangle\right\vert  \\
&= & \left\vert  \int_0^\infty \langle f_t,  t\Delta e^{-t\Delta} g\rangle\right\vert \\
&\leq & \left\Vert \left(\int_0^\infty |f_t|^2\frac{dt}{t}\right)^{1/2}\right \Vert_{q} \left\Vert \left(\int_0^\infty |t\Delta e^{-t\Delta}g|^2\frac{dt}{t}\right)^{1/2}\right \Vert_{q^{\prime}}\\
&\lesssim &\left \Vert \left(\int_0^\infty |f_t|^2\frac{dt}{t}\right)^{1/2}\right \Vert_{q} ||g||_{q^{\prime}}\\
& = &\left \Vert \left(\int_0^\infty |f_t|^2\frac{dt}{t}\right)^{1/2}\right \Vert_{q}.
\end{eqnarray*}
Note that since each ball $B_i$ is included in $2B$, the support of $f_t$ (in the $x$ variable) is included in $2B$. For all $x\in 2B$, set
\[
\mathcal{I}^x:=\left\{i;\ x\in B_i\right\}.
\]
\begin{comment}
Since, for all $x\in 2B$,
\[
\left\vert f_t(x)\right\vert \le \frac 1{r^{i}_\alpha}\sum_{i\in I^x_\alpha} \left\vert b^{i}_\alpha(x)\right\vert \mathbf{1}_{[(r^{i}_\alpha)^2,+\infty)}(t),
\]
\end{comment}
The fact that $\sharp \mathcal{I}^x\le N$ and the Cauchy-Schwarz inequality yield
\begin{eqnarray*}
\left\vert f_t(x)\right\vert^2 & \le & N\sum_{i\in \mathcal{I}^x}\frac{\left\vert b_i(x)\right\vert^2}{t} \mathbf{1}_{[r_i^2\wedge r_\beta^2,r_\beta^2)}(t)\leq N\sum_{i\in \mathcal{I}^x}\frac{\left\vert b_i(x)\right\vert^2}{t} \mathbf{1}_{[r_i^2\wedge  r_\beta^2,\infty)}(t)
\end{eqnarray*}
Since $B_i\subset 2B$, it follows that $r_i\leq 2r_\beta$, hence

$$\left\vert f_t(x)\right\vert^2 \leq N\sum_{i\in \mathcal{I}^x}\frac{\left\vert b_i(x)\right\vert^2}{t} \mathbf{1}_{\left[\frac{r_i^2}{4},\infty\right)}(t)\leq N\sum_{i\in I}\frac{\left\vert b_i(x)\right\vert^2}{t} \mathbf{1}_{\left[\frac{r_i^2}{4},\infty\right)}(t).$$
This entails
\begin{eqnarray*}
\int_0^\infty |f_t(x)|^2\frac{dt}{t}& \le & N\sum_{i\in I} \int_0^{+\infty} \frac{\left\vert b_i(x)\right\vert^2}{t} \mathbf{1}_{\left[\frac{r_i^2}{4},+\infty\right)}(t)\frac{dt}t\\
& \le & 4N\sum_{i\in I} \frac{\left\vert b_i(x)\right\vert^2}{r_i^2}.
\end{eqnarray*}
As a consequence,
\begin{eqnarray*}
\left\Vert \left(\int_0^\infty |f_t|^2\frac{dt}{t}\right)^{\frac 12}\right\Vert_{L^q}^q & \le & (4N)^{\frac q2} \int_{2B} \left(\sum_{i\in I} \frac{\left\vert b_i(x)\right\vert^2}{r_i^2}\right)^{\frac q2}d\mu(x)\\
& \le & (4N)^{\frac q2} \int_{2B} \sum_{i\in I} \frac{\left\vert b_i(x)\right\vert^q}{r_i^q}d\mu(x)\\
& = & (4N)^{\frac q2} \sum_{i\in I}\frac 1{r_i^q}\int_{B_i} \left\vert b_i(x)\right\vert^qd\mu(x)\\
& \lesssim & \sum_{i\in I} \int_{B_i}\left\vert \nabla \varphi(x)\right\vert^qd\mu(x)\\
& \lesssim & \left\Vert \nabla \varphi\right\Vert_q^q,
\end{eqnarray*}
\begin{comment}
$$\left \Vert \left(\int_0^\infty |b_t|^2\frac{dt}{t}\right)^{1/2}\right \Vert_{q}=\frac{1}{r^i_\alpha}||b^i_\alpha||_{q},$$
hence

$$\left\vert \int_{4B_\alpha} (U^i_\alpha b^i_\alpha )gd\mu\right\vert \leq \frac{1}{r_\alpha}||b^i_\alpha||_{q}\lesssim \left\Vert \nabla \varphi\right\Vert_{L^q(B^{i}_\alpha)},
$$
\end{comment}
where the fourth (resp. the fifth) line is derived from point $(3)$ (resp. point $(5)$) in Lemma \ref{lem:CZ}. 
%\todo{Les majorations pr\'ec\'edentes sont correctes si $q<2$. Si on veut traiter le cas $q\ge 2$, il faut utiliser $\left(\sum_\alpha a_\alpha\right)^{\frac q2}\le N^{\frac q2-1}\sum_\alpha a_\alpha^{\frac q2-1}$, ce qui donnera 
%\[
%\left(\sum_{i\in I} \frac{\left\vert b_i(x)\right\vert^2}{r_i^2}\right)^{\frac q2}\le N^{\frac q2-1} %\sum_{i\in I} \frac{\left\vert b_i(x)\right\vert^q}{r_i^q},
%\]
%et ne modifie la majoration qu'\`a une constante pr\`es. (Baptiste) On ne veut pas considérer le cas \Mg $q\geq 2$\Bk , cf encadré précédente. J'ai ajouté l'hypothèse $q\leq 2$ à l'énoncé du Lemme 3.9. \Bl Emmanuel: Tu veux dire: ``On ne veut pas considérer le cas \(q>2\) ?''. \Mg Oui c'est bien ca, le cas $q>2$ est inutile\Bk}
We thus have obtained the following estimate:

$$\left\vert \int_{M} \left(\sum_iU_i b_i\right)gd\mu\right\vert \lesssim ||\nabla \varphi||_q,\quad \forall g\in C_c^\infty(M),\,||g||_{q^\prime}=1$$
By $L^q-L^{q^\prime}$ duality, this implies that

$$\left\| \sum_iU_i b_i \right\|_q^q \lesssim ||\nabla \varphi||_q^q,$$
hence 

$$J\leq \mu\left(\left\{x\in M\,;\, \left\vert \sum_iU_i b_i (x)\right\vert > \frac{\lambda}{3}\right\}\right)\leq \lambda^{-q} \left\| \sum_iU_i b_i \right\|_q^q \lesssim \lambda^{-q}||\nabla \varphi||_q^q,$$
which is precisely \eqref{eq:estimUi}. This concludes the proof of Lemma \ref{lem:muh}.

\end{proof}

\bigskip

\noindent As a consequence of the weak type estimate provided by Lemma \ref{lem:muh}, we are now going to prove:
\begin{Lem}\label{lem:diagonal}
Let $ p\in (q,2)$. There is a constant $C>0$ such that for every $\alpha\in \N$,

$$ ||T_{\mathrm{diag}}f||_{L^p(4B_\alpha)}^p\leq C \sum_\beta \int_{B_\beta}|\nabla f_\beta|^p\,d\mu,$$
where the sum is taken over all $\beta$'s such that $4B_\alpha\cap 4B_\beta\neq\emptyset$.

\end{Lem}
\begin{proof}
We write

\begin{eqnarray*}
||T_{\mathrm{diag}}f||^p_{L^p(4B_\alpha)}&=&\left\|\sum_\beta \mathbf{1}_{4B_\alpha\cap 4B_\beta} \tilde{T}_\beta f_\beta\right\|_ p^p \\ &\leq & \left(\sum_\beta ||\tilde{T}_\beta f_\beta||_{L^p(4B_\beta)}\right)^p\\
&\lesssim  & \sum_\beta ||\tilde{T}_\beta f_\beta||_{L^p(4B_\beta)}^p,
\end{eqnarray*}
where the sums over $\beta$ are taken over all $\beta$'s such that $4B_\alpha\cap 4B_\beta\neq\emptyset$ and we have used in the last inequality that the number of terms in the sum is bounded by a constant which does not depend on $\alpha$. So, in order to conclude the proof, it is enough to show that, for all such $\beta$,
\begin{equation}\label{eq:Tdiag_est}
||\tilde{T}_\beta f_\beta||_{L^p(4B_\beta)}^p \leq C\int_{B_\beta}|\nabla f_\beta|^p\,d\mu.
\end{equation}
Let $C>0$ be given by Lemma \ref{lem:muh}, and denote
$$C_\beta=\left(\frac{C\left\Vert \nabla f_\beta\right\Vert_p^p}{\mu(B_\beta)}\right)^{1/p},$$
and $c=C^{1/q-1/p}$.

Let also $\Lambda>1$ be a large constant, whose value will be fixed later on. We first claim that
\begin{equation} \label{eq:estimI}
I:=\int^\infty_{c\Lambda  C_\beta} \lambda^{p-1}\mu(\{x\in 4B_\beta;\ |\tilde{T}_{\beta}f_\beta(x)|>\lambda\})\,d\lambda \lesssim  \int_{B_\beta} |\nabla f_\beta|^p\,d\mu,
\end{equation}
This estimate will be established through an interpolation type argument borrowed from \cite[Section 1.3]{AC}. Noticing that, since $f_\beta$ is supported in $B_\beta$ and by the H\"older inequality,
$$
 \left(C\frac{\left\Vert \nabla f_\beta\right\Vert_q^q}{\mu(B_\beta)}\right)^{\frac 1q}\le  c\left(\frac{C\left\Vert \nabla f_\beta\right\Vert_p^p}{\mu(B_\beta)}\right)^{\frac 1p}=cC_\beta,
$$
%\todo{Why do you consider the supremum ? I guess we mean
%\[
%\left(\frac{\left\Vert \nabla f_\beta\right\Vert_q^q}{\mu(B_\beta)}\right)^{\frac 1q}\leq \left(\frac{\left\Vert %\nabla f_\beta\right\Vert_p^p}{\mu(B_\beta)}\right)^{\frac 1p}=C_\beta.
%\]
%\textcolor{red}{Oui, je crois que c'est ca}
%}
so for any $\lambda> c\Lambda C_\beta$, using the Calder\'on-Zygmund decomposition (in $L^q$) given by Lemma \ref{lem:CZ} again with threshold $\bar{\lambda}=\frac{\lambda}{\Lambda}> cC_\beta \ge \left(\frac{C\left\Vert \nabla f_\beta\right\Vert_q^q}{\mu(B_\beta)}\right)^{\frac 1q}$, the function $f_\beta$ is decomposed as
$$
f_\beta=:g+b,
$$
(in the sequel the functions $g$, $b$ and the set $\Omega$ implicitly depend on $\beta$ and $\lambda$, but we choose not to make explicit these dependances in order to keep notations light), which yields
\begin{eqnarray*}
I & \le & \int^\infty_{c\Lambda C_\beta} \lambda^{p-1}\mu\left(\left\{x\in 4B_\beta;\ |\tilde{T}_\beta g(x)|>\frac{\lambda}2\right\}\right)\,d\lambda\\
& + & \int^\infty_{c\Lambda C_\beta} \lambda^{p-1}\mu\left(\left\{x\in 4B_\beta;\ |\tilde{T}_\beta b(x)|>\frac{\lambda}2\right\}\right)\,d\lambda\\
& \lesssim & \int_{0}^{\infty} \lambda^{p-1} \frac{\left\Vert \left\vert \nabla g\right\vert\right\Vert_2^2}{\lambda^2}d\lambda\\
& + & \int_0^{\infty} \lambda^{p-1} \frac{\left\Vert \left\vert \nabla b\right\vert\right\Vert_q^q}{\lambda^q}d\lambda\\
& =: & I_1+I_2.
\end{eqnarray*}
In the third and fourth lines, we used Equation \eqref{eq:weak_Ta} with the function $\varphi=b$, which has compact support in $2B_\beta$, as well as the inequality $||\tilde{T}_\beta g||_2\lesssim ||\nabla g||_2$ from Corollary \ref{cor:TbddL2} (this is again where we need the assumption $\nu>2$). Recall that the function $b$ is Lipschitz (since $f_\beta$ is smooth and $g$ is Lipschitz) and compactly supported in $2B_\beta$ by construction; moreover, we claim that indeed, Equation \eqref{eq:weak_Ta} with $\varphi=b$ is applicable. In fact, according to Equation \eqref{eq:est_nabla_b}, one has

$$\frac{||\nabla b||_q}{\mu(B_\beta)^{1/q}}\leq C'\bar{\lambda}=C'\frac{\lambda}{\Lambda},$$
therefore, one has $\lambda > \left(\frac{C\left\Vert \nabla b\right\Vert_q^q}{\mu(B_\beta)}\right)^{\frac 1q}$, provided $\Lambda>C'\cdot C^{\frac 1q}$, which shows that  Equation \eqref{eq:weak_Ta} with $\varphi=b$ is applicable if $\Lambda$ is chosen in this way. \Bk 
%\todo{Qu'est-ce que $C'$ ? $C'=C$ ? Cf Equation \eqref{eq:est_nabla_b}. (Baptiste) J'ai changé C en C' dans l'équation \eqref{eq:est_nabla_b} pour qu'il soit clair qui est C' \Bl Emmanuel: OK, mais je crois qu'il faut changer la condition sur \(\Lambda\), conformément à ce qui écrit juste avant l'encadré. \Mg OK ca me va \Bk}
Let us first estimate $I_1$. Lemma \ref{lem:CZ} yields
$$
\nabla g=\nabla f_\beta\cdot {\bf 1}_{M\setminus \Omega}+H
$$
where $H$ is a vector field, supported in $\Omega$ and $\left\Vert H\right\Vert_{\infty}\lesssim \lambda$. This decomposition provides, given that $f_\beta$ has support in $B_\beta$,

\begin{eqnarray}
I_1& \lesssim & \int_0^{\infty} \lambda^{p-1} \frac{\left\Vert \left\vert \nabla  f_\beta\right\vert\right\Vert_{L^2(B_\beta\setminus \Omega)}^2}{\lambda^2}d\lambda \nonumber\\
& + & \int_0^{\infty} \lambda^{p-1} \frac{\left\Vert H\right\Vert_2^2}{\lambda^2}d\lambda \nonumber \\
& =: & I_1^1+I_1^2 \label{eq:I11I12}.
\end{eqnarray}
On the one hand, since $p<2$,
\begin{eqnarray}
I_1^1 & \le & \int_0^{\infty} \lambda^{p-3} \left(\int_{B_\beta\setminus \Omega} \left\vert \nabla f_\beta\right\vert^2d\mu\right)d\lambda \nonumber\\
& \le & \int_{B_\beta} \left\vert \nabla f_\beta(x)\right\vert^2 \left(\int_{\left(\mathscr{M}\left(\left\vert \nabla f_\beta\right\vert^q\right)(x)\right)^{1/q}}^{\infty} \lambda^{p-3} d\lambda\right)d\mu(x)\nonumber\\
& \simeq & \int_{B_\beta} \left\vert \nabla f_\beta(x)\right\vert^2\left(\mathscr{M}\left(\left\vert \nabla f_\beta\right\vert^q\right)(x)\right)^{\frac{p-2}q}d\mu(x) \label{eq:estimI11},
\end{eqnarray} 
where, in order to pass from the first to the second line, we have used that, by definition of $\Omega$,
$$B_\beta\setminus \Omega=\{x\in B_\beta\,;\,\left(\mathscr{M}\left(\left\vert \nabla f_\beta\right\vert^q\right)(x)\right)^{1/q}\leq \lambda\}.$$
Since
\begin{eqnarray*}
\left\vert \nabla f_\beta(x)\right\vert^2 & = &\left\vert \nabla f_\beta(x)\right\vert^p \left\vert \nabla f_\beta(x)\right\vert^{2-p}\\
& \le & \left\vert \nabla f_\beta(x)\right\vert^p \left(\mathscr{M}\left(\left\vert \nabla f_\beta\right\vert^q\right)(x)\right)^{\frac{2-p}q} \quad \mbox{a. e. }x\in M,
\end{eqnarray*}
it follows from \eqref{eq:estimI11} that
\begin{equation} \label{eq:estimI11bis}
I_1^1\lesssim \int_{B_\beta} \left\vert \nabla f_\beta(x)\right\vert^pd\mu(x).
\end{equation}
%\todo{Here, we have twice the same term ? \textcolor{red}{corrige, ca doit etre une typo}}
On the other hand, since $H$ is supported in $\Omega\subset 2B_\beta$ and $\left\vert H\right\vert\lesssim \lambda$, we get by using the definition of $\Omega$ that
\begin{eqnarray}
I_1^2 & \le & \int_0^{\infty} \lambda^{p-1}  \mu(\Omega)d\lambda \nonumber\\
& = & \int_0^{\infty} \lambda^{p-1}  \left(\int_{2B_\beta} {\bf 1}_{\Omega}(x)d\mu(x)\right)d\lambda\nonumber\\
& \le & \int_{2B_\beta} \left(\int_0^{\left(\mathscr{M}\left(\left\vert \nabla f_\beta\right\vert^q\right)(x)\right)^{1/q}} \lambda^{p-1}d\lambda\right)d\mu(x)\nonumber\\
& \lesssim & \int_{2B_\beta} \left(\mathscr{M}\left(\left\vert \nabla f_\beta\right\vert^q\right)(x)\right)^{p/q}d\mu(x)\nonumber\\
& \le & \left\Vert \mathscr{M}\left(\left\vert \nabla f_\beta\right\vert^q\right)\right\Vert_{p/q}^{p/q}\nonumber\\
& \lesssim & \left\Vert \left\vert \nabla f_\beta\right\vert^q\right\Vert_{p/q}^{p/q}\nonumber\\
& = & \int_{B_\beta} \left\vert \nabla f_\beta(x)\right\vert^pd\mu(x) \label{eq:estimI12},
\end{eqnarray}
where the sixth line holds since $\frac pq>1$. Gathering \eqref{eq:I11I12}, \eqref{eq:estimI11bis} and \eqref{eq:estimI12} shows that
\begin{equation} \label{eq:estimI1}
I_1\lesssim \int_{B_\beta} \left\vert \nabla f_\beta(x)\right\vert^pd\mu(x).
\end{equation}
Our next task is to estimate $I_2$. 

%To that purpose, using Lemma \ref{lem:CZ} again, one starts from
%$$
%\nabla b=\nabla f_\beta-\nabla g=\nabla f_\beta\cdot {\bf 1}_{\Omega}-H,
%$$
According to Proposition \ref{pro:Lq_CV_nabla_bi}, $||\nabla b||_q\lesssim \lambda \mu(\Omega)^{1/q}$, hence

$$I_2\lesssim \int_0^\infty \lambda^{p-1}\mu(\Omega)\,d\lambda,$$
an integral which we have already shown is estimated by $\int_{B_\beta}|\nabla f_\beta|^p\,d\mu$ in the course of our estimations of $I_1^2$. This concludes the proof of \eqref{eq:estimI}. \par

\bigskip

\noindent Let us consider now

$$J:=\int_0^{c\Lambda C_\beta}\lambda^{p-1}\mu(\{x\in 4B_\beta\,;\,|\tilde{T}_\beta f_\beta(x)|>\lambda\})\,d\lambda.$$
\noindent Using a trivial estimate and doubling, one obtains at once

$$\mu(\{x\in 4B_\beta\,;\,|\tilde{T}_\beta f_\beta(x)|>\lambda\})\leq  \mu(4B_\beta)\lesssim \mu(B_\beta),$$
so that

\begin{equation} \label{eq:estimJ}
J\lesssim \mu(B_\beta)\int_0^{c\Lambda\left(\frac{C\left\Vert \nabla f_\beta\right\Vert_p^p}{\mu(B_\beta)}\right)^{1/p}}\lambda^{p-1}\,d\lambda= \frac{c^pC\Lambda^p}p  \int_{B_\beta}|\nabla f_\beta|^p.
\end{equation}
Putting together \eqref{eq:estimI} and \eqref{eq:estimJ}, we get
\begin{eqnarray*}
 ||\tilde{T}_\beta f_\beta||_{L^p(4B_\beta)}^p &=& p\int_0^\infty \lambda^{p-1} \mu(\{x\in 4B_\beta\,;\,|\tilde{T}_\beta f_\beta(x)|>\lambda\})\,d\lambda\\
&\leq & p(I+J)\\
&\lesssim&  \int_{B_\beta}|\nabla f_\beta|^p,
\end{eqnarray*}
which ends the proof of \eqref{eq:Tdiag_est}, and hence the proof of Lemma \ref{lem:diagonal}.
\end{proof}

%Then, write
%
%$$\Delta^{1/2}f_\alpha=\int_0^{r_\alpha^2} \frac{\partial}{\partial t}e^{-t\Delta}f_\alpha\,\frac{dt}{\sqrt{t}} +\int_{r_\alpha^2}^\infty \frac{\partial}{\partial t}e^{-t\Delta}f_\alpha\,\frac{dt}{\sqrt{t}} = T_\alpha f_\alpha + U_\alpha f_\alpha.$$
%We are going to estimate $||T_\alpha f_\alpha||_{L^p(4B_\alpha)}$, $||T_\alpha f_\alpha||_{L^p(M\setminus 4B_\alpha)}$ and $||U_\alpha f_\alpha||_p$. 
%
%Let us now start the proof of Theorem \ref{thm:main}. Let $f\in C_0^\infty(M)$. As mentionned above, we write 
%
%$$f=\sum_{\alpha\in\N}\chi_\alpha f=\sum_{\alpha\in\N} f_\alpha.$$
%We first estimate $\left\Vert \Delta^{1/2}f_{\alpha}\right\Vert_{L^p(4B_\alpha)}$. 
\begin{Lem} \label{lem:sumdiagonal}
For all $p\in (q,2)$, there exists $C>0$ such that
\[
\left\Vert T_{\mathrm{diag}}f\right\Vert_p\le C \left(\left\Vert \nabla f\right\Vert_p+\left\Vert \frac f{1+r}\right\Vert_p\right).
\]
\end{Lem}
\begin{proof}
By the finite overlap property for the balls $4B_\alpha$, there exists $L>0$ such that, for all $x\in M$,

$$1\leq \sum_{\alpha}\mathbf{1}_{4B_\alpha}(x)\leq L<+\infty,$$
hence

\[
|T_{\mathrm{diag}}f|^p(x)\leq \left(\sum_\alpha {\bf 1}_{4B_\alpha}(x)\left\vert T_{\mathrm{diag}}f(x)\right\vert\right)^p\le L^{\frac p{p^{\prime}}}\sum_\alpha {\bf 1}_{4B_\alpha}(x)\left\vert T_{\mathrm{diag}}f(x)\right\vert^p. 
\]
Integration on $M$ therefore yields
\begin{eqnarray*}
\int_M |T_{\mathrm{diag}}f|^p(x) d\mu(x) &\le & L^{\frac p{p^{\prime}}} \sum_\alpha\int_M {\bf 1}_{4B_\alpha}(x)\left\vert T_{\mathrm{diag}}(x)\right\vert^pd\mu(x)\\
& \lesssim &  \sum_{\alpha} \sum_\beta \int_{B_\beta} \left\vert \nabla f_\beta(x)\right\vert^pd\mu(x),
\end{eqnarray*}
where, for all $\alpha$, the sum in $\beta$ is taken over all $\beta$'s such that $4B_\alpha\cap 4B_\beta\neq \emptyset$, as follows from Lemma \ref{lem:diagonal}. However, this sum in $\beta$ has at most $L$ non-vanishing terms by definition of $L$ above, hence

$$\int_M |T_{\mathrm{diag}}f|^p(x) d\mu(x)\lesssim \sum_{\alpha}  \int_{B_\alpha} \left\vert \nabla f_\alpha(x)\right\vert^pd\mu(x).$$
But, for all $x\in B_\alpha$,
\begin{eqnarray*}
\left\vert \nabla f_\alpha(x)\right\vert^p & \le & 2^{p-1}\left(\left\vert \nabla f(x)\right\vert^p+\left\vert f(x)\right\vert^p\left\vert \nabla \chi_\alpha(x)\right\vert^p\right)\\
& \le & 2^{p-1}\left(\left\vert \nabla f(x)\right\vert^p+\frac{\left\vert f(x)\right\vert^p}{r_\alpha^p}\right),
\end{eqnarray*}
so that
\begin{eqnarray*}
\int_{B_\alpha} \left\vert \nabla f_\alpha(x)\right\vert^pd\mu(x) & \le & 2^{p-1}\left(\int_{B_\alpha} \left\vert \nabla f(x)\right\vert^pd\mu(x)+\int_{B_\alpha} \frac{\left\vert f(x)\right\vert^p}{r_\alpha^p}d\mu(x)\right)\\
&  \lesssim  & 2^{p-1}\left(\int_{B_\alpha} \left\vert \nabla f(x)\right\vert^pd\mu(x)+\int_{B_\alpha} \frac{\left\vert f(x)\right\vert^p}{(1+r(x))^p}d\mu(x)\right).
\end{eqnarray*}
Using the bounded overlap property for the balls $B_\alpha$ again, one concludes
\[
\int_M |T_{\mathrm{diag}}f(x)|^pd\mu(x) \le C\left(\int_M \left\vert \nabla f(x)\right\vert^pd\mu(x)+\int_M \frac{\left\vert f(x)\right\vert^p}{(1+r(x))^p}d\mu(x)\right),
\]
which concludes the proof.
\end{proof}
\subsection{Estimates of the non-diagonal terms} It remains to study the ``off-diagonal'' term $T_{\mathrm{off}}$, which is our next task. Recall its definition:

$$T_{\mathrm{off}}f:=\sum_{\alpha\in \N} \mathbf{1}_{M\setminus 4B_\alpha}\tilde{T}_\alpha f_\alpha.$$

\begin{Pro}\label{pro:Ta}

Let $p\in(1,\infty)$. Then, the operator $T_{\mathrm{off}}$ satisfies: 

\begin{equation}\label{eq:T}
||T_{\mathrm{off}}f||_p\lesssim \left\|\frac{f}{1+r}\right\|_p.
\end{equation}

\end{Pro}
The proof of Proposition \ref{pro:Ta} is split into a sequence of two lemmata. In the proof, it will be convenient to call

$$T_0f=\mathbf{1}_{M\setminus 4B_0} \tilde{T}_0f_0,$$
and

$$T_{>}f=\sum_{\alpha\neq 0}\mathbf{1}_{M\setminus 4B_\alpha}\tilde{T}_\alpha f_\alpha.$$

\begin{Lem}\label{lem:kernel}

Let $k_>(x,y)=\mathbf{1}_{r(y)\geq r_0}\mathbf{1}_{d(x,y)\geq C^{-1} r(y)}\int_0^{Cr^2(y)}|\partial_tp_t(x,y)|\,\frac{dt}{\sqrt{t}}$, and $k_0(x,y)=\mathbf{1}_{r(x)\geq  4r_0}\mathbf{1}_{B_0}(y)\int_0^{r_0^2}|\partial_tp_t(x,y)|\,\frac{dt}{\sqrt{t}}$. Then, if the constant $C>0$ is large enough, only depending on the universal constants involved in the covering provided by Proposition \ref{prop:goodcovering}, one has for every $f\in C_0^\infty(M)$ and $g\geq 0$,
$$\int_M|T_{>}f|(x)g(x)\,dx\leq N \int_{M\times M}g(x)k_>(x,y)|f(y)|\,dxdy,$$
where $N<\infty$ is the maximal number of intersecting balls in the covering $(B_\alpha)_{\alpha\in\N}$, see Proposition \ref{prop:goodcovering}. Moreover, for every $f\in C_0^\infty(M)$ and $g\geq 0$,

$$\int_M|T_{0}f|(x)g(x)\,dx\leq \int_{M\times M}g(x)k_0(x,y)|f(y)|\,dxdy. $$
%\todo{It seems to me that $c$ is not needed, it is not involved in the definition of $k_0$. \textcolor{red}{OK, ca %doit etre une ancienne version. Par contre tu as remplace un $c$ par $C^{-1}$ dans $k_>$ et il faut corriger la %preuve du Lemm 3.11 pour incorporer ce changement }}
\end{Lem}

\begin{proof}

Let us start with $T_>$. Let $y\in B_\alpha$, $\alpha\neq 0$. Then, point (5) from Proposition \ref{prop:goodcovering} implies that

\begin{equation}\label{eq:comp_r}
(2^9-1)r_\alpha\leq r(y)\leq (2^{10}+1)r_\alpha.
\end{equation}
Thus, 

\begin{equation}\label{eq:tleqr}
r_\alpha^2\le C_1r^2(y)
\end{equation}
with $C_1=(2^9-1)^{-2}$. And moreover, if $x\in M\setminus 4B_\alpha$, then
\begin{equation}\label{eq:dleqr}
d(x,y)\geq C^{-1}_2r(y),
\end{equation}
with $C_2=(2^{10}+1)$. In what follows, we let $C=\max(C_1,C_2)$. On the other hand, by Fubini,
\begin{eqnarray*}
\int_M |T_>f(x)|g(x)\,dx & \leq & \int_{M\times M}\left(\sum_{\alpha\neq 0}\mathbf{1}_{M\setminus 4B_\alpha}(x)\chi_\alpha(y)\int_0^{r^2_\alpha}|\partial_tp_t(x,y)|\,\frac{dt}{\sqrt{t}}\right)\\
 & & \times g(x)|f(y)|\,dxdy.
 \end{eqnarray*}
Fix $y\in M$, and let $J(y)=\{\alpha\neq 0\,;\,y\in B_\alpha\}$. Then, by (5) of Proposition \ref{prop:goodcovering},

$$\sharp J(y)\leq N.$$
Moreover, according to \eqref{eq:tleqr} and \eqref{eq:dleqr}, and the fact that $\alpha=0$ is excluded,

\begin{eqnarray*}
\mathbf{1}_{M\setminus 4B_\alpha}(x)\chi_\alpha(y)\int_0^{r^2_\alpha}|\partial_tp_t(x,y)|\,\frac{dt}{\sqrt{t}}&\leq  &  \mathbf{1}_{d(x,y)\geq C^{-1}r(y)}\mathbf{1}_{r(y)\geq r_0}\\
& &\times \int_0^{Cr^2(y)}|\partial_tp_t(x,y)|\,\frac{dt}{\sqrt{t}}.
\end{eqnarray*}
Summing over $\alpha\in J(y)$, one obtains that one can bound the term

$$\sum_{\alpha\neq 0}\mathbf{1}_{M\setminus 4B_\alpha}(x)\chi_\alpha(y)\int_0^{r^2_\alpha}|\partial_tp_t(x,y)|\,\frac{dt}{\sqrt{t}}$$
by $k_>(x,y)$, which is the desired result for $T_>$. The proof for $T_0$ is easy and is skipped.

\end{proof}

\begin{Lem}\label{lem:Kp}

Denote by $K$ the integral operator, the kernel of which is given by $k_>+k_0$. Then, for every $p\in (1,\infty)$,

$$||Kf||_p\lesssim \left\|\frac{f}{1+r}\right\|_p.$$

\end{Lem}

\begin{proof}[Proof of Proposition \ref{pro:Ta}:]

Given that $T=T_0+T_>$ by definition, the result follows right away from Lemma \ref{lem:kernel} and Lemma \ref{lem:Kp}.

\end{proof}

\begin{proof}[Proof of Lemma \ref{lem:Kp}:]

Let $K_>$ and $K_0$ be the integral operators, the kernels of which are $k_>(x,y)$ and $k_0(x,y)$ respectively. We show that

\begin{equation}\label{eq:K_>}
||K_>f||_p\lesssim  \left\|\frac{f}{1+r}\right\|_p
\end{equation}
and

\begin{equation}\label{eq:K_0}
||K_0f||_p\lesssim  \left\|\frac{f}{1+r}\right\|_p.
\end{equation}
Let us start with \eqref{eq:K_>}. We estimate the kernel $k_>(x,y)$: first, by \eqref{eq:DV}, for every $t\leq Cr^2(y)$, and $d(x,y)\geq C^{-1} r(y)$,
\begin{eqnarray}
\frac{1}{V(y,\sqrt{t})} & = & \frac{V(y,r(y))}{V(y,\sqrt{t})}\frac{1}{V(y,r(y))}\nonumber\\
&\lesssim & \left(\frac{r^2(y)}{t}\right)^{D/2}\frac{1}{V(o,r(y))}\nonumber\\
&\lesssim & \left(\frac{d^2(x,y)}{t}\right)^{D/2}\left(\frac{r(y)}{d(x,y)}\right)^D\frac{1}{V(o,r(y))} ,\label{eq:volume_k}
\end{eqnarray}
where we have used $V(y,r(y))\simeq V(o,r(y))$ to go from the first to the second line. Then, by \eqref{eq:dUE2} and \eqref{eq:volume_k}, we get

\begin{equation}\label{eq:estim_k_>}
k_>(x,y)\lesssim \mathbf{1}_{r(y)\geq r_0}\mathbf{1}_{d(x,y)\geq C^{-1} r(y)}\left(\frac{r(y)}{d(x,y)}\right)^D\frac{1}{V(o,r(y))}\int_0^{Cr^2(y)}e^{-c'\frac{d^2(x,y)}{t}}\,\frac{dt}{t^{3/2}}.
\end{equation}
The integral in the above formula is easily estimated for $d(x,y)\geq C^{-1} r(y)$:

\begin{eqnarray*}
\int_0^{Cr^2(y)}e^{-c'\frac{d^2(x,y)}{t}}\,\frac{dt}{t^{3/2}} & \leq & \frac{1}{d(x,y)}\int_{C^{-1}\frac{d^2(x,y)}{r^2(y)}}^\infty e^{-c'u}\,\frac{du}{\sqrt{u}} \\
& \leq & \frac{1}{d(x,y)} e^{-\frac{c'}{2}\frac{d^2(x,y)}{r^2(y)}}\int_{C^{-3}}^{\infty}  e^{-\frac{c'}{2}u}\,\frac{du}{\sqrt{u}} \\
&\leq & \frac{c}{r(y)}e^{-\frac{c'}{2}\frac{d^2(x,y)}{r^2(y)}} \times \mathrm{const.}
\end{eqnarray*}
Thus, we deduce from \eqref{eq:estim_k_>} the following estimate on $k_>(x,y)$ :

\begin{equation}\label{eq:estim_k_>2}
k_>(x,y)\lesssim \mathbf{1}_{r(y)\geq r_0}\mathbf{1}_{d(x,y)\geq C^{-1}r(y)} \frac{e^{-C\frac{d^2(x,y)}{r^2(y)}}}{r(y)V(o,r(y))},
\end{equation}
where $C>0$ is some small enough constant. Denote 

$$\tilde{k}(x,y):=(1+r(y))k_>(x,y),$$
let $\tilde{K}$ be the integral operator associated to $\tilde{k}$, and note that \eqref{eq:K_>} is equivalent to the fact that $\tilde{K} : L^p\to L^p$ is bounded. The estimate \eqref{eq:estim_k_>2} on $k_>(x,y)$ implies readily the following estimate for $\tilde{k}(x,y)$:

\begin{equation}\label{eq:estim_k_t}
\tilde{k}(x,y)\lesssim \mathbf{1}_{d(x,y)\geq C^{-1}r(y)} \frac{e^{-C\frac{d^2(x,y)}{r^2(y)}}}{V(o,r(y))}.
\end{equation}
Now, in order to show that $\tilde{K} : L^p\to L^p$ is bounded, we show equivalently that its adjoint $\tilde{K}^*$ is $L^q\to L^q$ bounded, $\frac{1}{p}+\frac{1}{q}=1$. \par
\noindent The adjont $\tilde{K}^*$ is given by
\[
\tilde{K}^*f(x)=(1+r(x))\int_M k_>(y,x)f(y)dy,
\]
which entails
\begin{eqnarray*}
\left\vert \tilde{K}^*f(x)\right\vert & \le & \frac{1+r(x)}{r(x)}\frac 1{V(o,r(x))}{\bf 1}_{r(x)\ge r_0} \int_M {\bf 1}_{d(y,x)\ge C^{-1}r(x)} e^{-C\frac{d^2(y,x)}{r^2(x)}}\left\vert f(y)\right\vert dy\\
& \le & \frac 1{V(o,r(x))}\int_M {\bf 1}_{d(y,x)\ge C^{-1}r(x)} e^{-C\frac{d^2(y,x)}{r^2(x)}}\left\vert f(y)\right\vert dy.
\end{eqnarray*}
This shows that, if
\[
k^\sharp(x,y):=\frac 1{V(o,r(x))} {\bf 1}_{d(y,x)\ge C^{-1}r(x)} e^{-C\frac{d^2(y,x)}{r^2(x)}},
\]
the kernel of $\tilde{K}^\ast$ is pointwise dominated by $k^\sharp$. 
\begin{comment}
The adjoint $\tilde{K}^*$ is an integral operator whose kernel $k^*(x,y)$ is given by $k^*(x,y)=\tilde{k}_>(y,x)$, hence the following estimate:

\begin{equation}\label{eq:estim_k*}
k^*(x,y)\lesssim \mathbf{1}_{d(x,y)\geq C^{-1}r(x)} \frac{e^{-C\frac{d^2(x,y)}{r^2(x)}}}{V(o,r(x))}.
\end{equation}
\end{comment}
\noindent For $x\in M$, let us estimate $||k^\sharp(x,\cdot)||_p^p$; one has

\begin{eqnarray*}
||k^\sharp(x,\cdot)||_p^p & = & \frac{1}{V(o,r(x))^p}\int_{d(x,y)\geq C^{-1}r(x)} e^{ -C^\prime\frac{d^2(x,y)}{r^2(x)}}\,dy\\
& = & \frac{1}{V(o,r(x))^p} \sum_{j\geq 0} \int_{2^jC^{-1} < \frac{d(x,y)}{r(x)}\leq 2^{j+1}C^{-1}} e^{-C^\prime\frac{d^2(x,y)}{r^2(x)}}\,dy\\
& \leq & \frac{1}{V(o,r(x))^p} \sum_{j\geq 0} V(x,2^{j+1}C^{-1} r(x))e^{-c'4^j} \\
& \lesssim & \frac{1}{V(o,r(x))^p} V(x,r(x))\sum_{j\geq 0} 2^{(j+1)D}e^{-c'4^j}\\
& \lesssim & V(o,r(x))^{1-p},
\end{eqnarray*} 
where we have used \eqref{eq:DV} and the fact that $V(x,r(x))\simeq V(o,r(x))$. As a consequence,

$$||k^\sharp(x,\cdot)||_p\lesssim \frac{1}{V(o,r(x))^\frac{1}{q}}.$$
This implies by H\"older, that for $f\in C_c^\infty(M)$, 

\begin{eqnarray*}
|\tilde{K}^* f(x)| & \leq & \int_M k^\sharp(x,y)|f(y)|\;dy \\
&\leq & ||k^\sharp(x,\cdot)||_p\cdot ||f||_q \\
& \lesssim & \frac{||f||_q}{V(o,r(x))^\frac{1}{q}}.
\end{eqnarray*}
Let $\lambda>0$ et $x\in M$ be such that $|\tilde{K}^* f(x)|>\lambda$, then one gets from the above estimate that

$$V(o,r(x))\lesssim \frac{||f||_q^q}{\lambda^q}.$$
According to \cite[Lemma (3.2)]{DR}, one concludes that

$$\mu\left(\{x\in M\,;\,| \tilde{K}^* f(x)|>\lambda \}\right)\lesssim \frac{||f||_q^q}{\lambda^q},$$
that is, the operator $\tilde{K}^* $ is bounded from $L^q$ into $L^q_w$, for every $q\in (1,\infty)$. By interpolation, we conclude that $\tilde{K}^*$ is bounded from $L^q$ into $L^q$ for every $q\in (1,\infty)$, and finally, by duality, that $\tilde{K}$ is bounded on $L^p$, for every $p\in (1,\infty)$. Consequently, \eqref{eq:K_>} holds.

It remains to prove \eqref{eq:K_0}. For that, we first have the following estimate for $k_0(x,y)$, which follows from \eqref{eq:dUE2}:

\begin{equation}\label{eq:estim_k_0}
k_0(x,y)\lesssim \mathbf{1}_{r(x)\geq 4r_0} \mathbf{1}_{B_0}(y)e^{-cr^2(x)},
\end{equation}
for some small enough constant $c>0$. Indeed, for all $x\notin 4B_0$ and all $y\in B_0$, $d(x,y)\ge \frac 34 r(x)$, which entails
\begin{eqnarray*}
\left\vert \partial_tp_t(x,y)\right\vert & \lesssim & \frac 1t \frac 1{V(y,\sqrt{t})}\exp\left(-c\frac{d^2(x,y)}t\right)\\
& \lesssim & \frac 1t \frac 1{V(y,r_0)}\left(\frac{r_0}{\sqrt{t}}\right)^D\exp\left(-c\frac{r^2(x)}t\right)\\
& \lesssim & \frac 1t \frac 1{V(o,r_0)}\exp\left(-\frac c2\frac{r^2(x)}t\right),
\end{eqnarray*}
where we used doubling and $r_0\le r(x)$. As a consequence,
\begin{eqnarray*}
k_0(x,y)& \lesssim & \mathbf{1}_{r(x)\geq 4r_0} \mathbf{1}_{B_0}(y)\frac 1{V(o,r_0)} \int_0^{r_0^2} \exp\left(-\frac c2\frac{r^2(x)}t\right) \frac 1{t^{\frac 32}}dt\\
& \le &  \mathbf{1}_{r(x)\geq 4r_0} \mathbf{1}_{B_0}(y)\frac 1{V(o,r_0)}\frac 1{r_0}\exp\left(-\frac c4\frac{r^2(x)}{r_0^2}\right)\int_1^\infty \exp\left(-\frac c4\frac{r^2(x)}{r_0^2}u\right) \frac {du}{\sqrt{u}}\\
& \lesssim & \mathbf{1}_{r(x)\geq 4r_0} \mathbf{1}_{B_0}(y) \exp\left(-\frac c4\frac{r^2(x)}{r_0^2}\right).
\end{eqnarray*}
Hence, if as above one lets $\tilde{k}_0(x,y)=(1+r(y))k_0(x,y)$ and $\tilde{K}_0$ is the integral operator with kernel $\tilde{k}_0$, the kernel of $\tilde{K}_0^\ast$ is pointwise dominated by
\[
k_0^\sharp(x,y):=\mathbf{1}_{r(y)\geq 4r_0} \mathbf{1}_{B_0}(x)e^{-cr^2(y)},
\]
and it follows from an easy application of H\"older's inequality that

$$\left\Vert \tilde{K}^*_0f\right\Vert_\infty \lesssim ||f||_q,\quad \frac{1}{q}+\frac{1}{p}=1.$$
Since $\tilde{K}^*_0f$ has support in $B_0$, one concludes that

$$||\tilde{K}_0^*f||_q\leq \mu(B_0)^{1/q}||\tilde{K}^*_0f||_\infty \lesssim ||f||_q.$$
Thus, $\tilde{K}_0^*$ is bounded on $L^q$, $q\in (1,\infty)$. We conclude that $\tilde{K}_0$ is bounded on $L^p$ and then that \eqref{eq:K_0} holds as in the proof of \eqref{eq:K_>}. This concludes the proof of the lemma.

%\begin{eqnarray*}
%|K_0f(x)|^p\,dx

%\end{eqnarray*}

\end{proof}

\subsection{Conclusion of the proof}

We now have all the ingredients for the conclusion of the proof of Theorem \ref{thm:main}. By definition of the operators, the following identity holds:

$$\Delta^{1/2}=c(U+T_{\mathrm{diag}}+T_{\mathrm{off}})$$
According to Lemma \ref{lem:Uterm}, Lemma \ref{lem:sumdiagonal} and Proposition \ref{pro:Ta}, we obtain: for every $f\in C_c^\infty(M)$, and every $1<p<\nu$,

$$||\Delta^{1/2}f||_p\lesssim \left\|\frac{f}{1+r}\right\|_p+||\nabla f||_p.$$
The Hardy inequality from Theorem \ref{thm:hardy} now implies that 

$$||\Delta^{1/2}f||_p\lesssim ||\nabla f||_p,$$
hence ($RR_p$).

\section{Appendix: proof of the Calder\'on-Zygmund lemma for Sobolev functions}

In this section we explain the proof of Lemma \ref{lem:CZ}; the construction of the (Whitney type) covering $(B_i)_{i\in \N}$ satisfying (5), (6) and (7), and of the functions $b_i$, has already been presented in details in \cite[Appendix B]{DR}. Here we intend to explain mainly the proof of points 2-4. We assume also that $\Omega\neq \emptyset$, otherwise the Calder\'on-Zygmund decomposition simply writes $g=u$. We denote by $F$ the complement of $\Omega$ in $M$. The proof of (4) is easy enough: according to (5), one has

$$\sum_{i\geq 1} V(B_i)\leq N \mu(\Omega)\lesssim \lambda^{-q}\int|\nabla u|^qd\mu,$$
where in the last inequality we have used the weak (1,1) type of the maximal function and the definition of $\Omega$. This proves (4). Let us now recall how the functions $b_i$ are defined: according to \cite{DR}, one can find a smooth partition of unity $(\chi_i)_{i\in \N}$ associated with the covering $(B_i)_{i\in \N}$ of $\Omega$, and such that for every $i\in\N$,

$$||\nabla\chi_i||_{L^\infty}\lesssim \frac{1}{r_i},$$
where $r_i$ denotes the radius of $B_i$. Then, $b_i$ is defined by

$$b_i=(u-u_{B_i})\chi_i.$$
It is clear by definition that $b_i$ has support in $B_i$. Moreover, since $B_i\subset \Omega\subset 2B$ (see \cite{DR}), it follows that $B_i$ satisfies the $L^q$ Poincar\'e inequality. Hence,

\begin{equation} \label{eq:bilq}
||b_i||_q\leq \left(\int_{B_i}|u-u_{B_i}|^q\,d\mu\right)^{1/q}\lesssim r_i ||\nabla u||_{L^q(B_i)}.
\end{equation}
Also,

$$\nabla b_i=(u-u_{B_i})\nabla \chi_i+\chi_i\nabla u,$$
so that, again applying Poincar\'e on $B_i$ and the estimate on $\nabla \chi_i$, we obtain

$$||\nabla b_i||_q\lesssim ||\nabla u||_{L^q(B_i)}.$$
But property (7) in Lemma \ref{lem:CZ} and doubling imply that

\begin{eqnarray*}
||\nabla u||^q_{L^q(B_i)}&\leq& ||\nabla u||^q_{L^q(3B_i)}\\
&\leq & V(3B_i)\lambda^q \\
&\lesssim & V(B_i)\lambda^q,
\end{eqnarray*}
so (3) holds. Define 

$$b=\sum_{i\geq 0} b_i,$$
and let 

$$g=u-b.$$
We will see in a moment that $b$ is actually a well-defined, locally integrable function on $M$. Since $u$ has support in $B$ and $b$ in $2B$, it follows that $g$ has support in $2B$. It remains to prove (2). Since the covering is locally finite by (5), the sum defining $b$ is merely a finite sum at every point in $\Omega$. There is a subtle point which is that it is possible that the balls $B_i$ accumulate near the boundary of $\Omega$, making $\nabla b$ having a singularity on the boundary of $\Omega$ (think of the extreme case where $b=\mathbf{1}_\Omega$, for instance). So, despite the fact that each $b_i$ is smooth and has support inside $\Omega$, and despite the sum $\sum_{i\geq 0}b_i$ being locally finite in $\Omega$, one must check carefully that $b$ is Lipschitz up to the boundary of $\Omega$. First, let us see that the series defining $b$ converges in $L^1_{loc}(M)$. Indeed, let $K$ be a compact set in $M$, and $\varphi\in L^\infty(M)$ vanishing outside of $K$; then, for every $n\in \N$,

\bean
\langle \sum_{i\leq n}|b_i|,\varphi\rangle & = & \sum_{i\leq n} \langle |b_i|,\varphi\rangle\\
& = &    \sum_{i\leq n} \langle \frac{|b_i|}{r_i},r_i \varphi\rangle\\
&\leq & \sum_{i\leq n} \langle \frac{|b_i|}{r_i},r_i |\varphi|\rangle\\
& \leq & \sum_{i\leq n} \left|\left|\frac{b_i}{r_i}\right|\right|_q\sup_{x\in K}d(x,F) ||\varphi||_{L^{q'}}\\
&\lesssim  &  \sum_{i\leq n} ||\nabla u||_{L^q(B_i)} ||\varphi||_\infty\\
&\lesssim & ||\nabla u||_{q}||\varphi||_\infty.
\eean
(in the last line, we have used the inequality $\sum_{i=1}^\infty ||\nabla u||_{L^q(B_i)}\lesssim ||\nabla u||_q$). Since $n$ is arbitrary, this proves that $\sum_{i\geq 0} |b_i|$ converges in $L^1_{loc}$, hence $b\in L^1_{loc}$ is well-defined. This yields that $g\in L^1_{loc}(M)$, too. In fact, since all the functions $b_i,b,g$ have by construction support in $2B$, the convergence of the two series is in $L^1(M)$, and $g\in L^1(M)$. The estimate on $||\nabla b_i||_q$ and the fact that the covering satisfies (4) in Lemma \ref{lem:CZ} easily imply that $\nabla b$, defined as a distribution, actually belongs to $L^q(M)$, and one has the following equality in $L^q$:

$$\nabla b=\sum_{i\geq 0}\nabla b_i=\sum_{i\geq 0} \left((u-u_{B_i})\nabla \chi_i+(\nabla u)\cdot \chi_i\right).$$
(see also Proposition \ref{pro:Lq_CV_nabla_bi}). It is clear that $\sum_{i\geq 0} (\nabla u)\cdot \chi_i$ converges in $L^q$ to $(\nabla u)\cdot \mathbf{1}_\Omega,$ hence $\nabla g$, defined as a distribution, actually belongs to $L^q(M)$ and we get the following equality in $L^q(M)$:

$$\nabla g=(\nabla u)\cdot \mathbf{1}_{F}-\sum_{i\geq 0} (u-u_{B_i})\nabla \chi_i.$$
Define 

$$H=-\sum_{i\geq 0} (u-u_{B_i})\nabla \chi_i,$$
which is an $L^q$ vector field since the series of the $L^q$ norms converge. We claim that the vector field $H$ is in fact essentially bounded, and that we have the estimate $||H||_{L^\infty} \lesssim \lambda$. This is proven in \cite{A} and we partially reproduce the proof from there, adding some more details. 

%We start by noticing that since the series defining $h$ converges in $L^q$, one can find a sequence $n_k\to +\infty$ such that a.e.,

%$$h=-\lim_{n_k\to\infty} \sum_{i\leq n_k} (u-u_{B_i})\nabla \chi_i.$$
\noindent Since $L^1(TM)\cap L^{q'}(TM)$ is dense in $L^1(TM)$, it is enough to prove that for every vector field $X\in L^1\cap L^{q'}(TM)$,

$$|\langle H,X\rangle |\lesssim \lambda ||X||_1.$$
Here, $q'$ denotes the conjugate exponent to $q$, that is $\frac{1}{q}+\frac{1}{q'}=1$.
%Since $\varphi\in L^{q'}$ and the series defining $h$ converges in $L^q$, it is enough to prove that there is a constant $C>0$ such that for every $n\in \N$, 

%$$\left| \langle \sum_{i\leq n} (u-u_{B_i})\nabla \chi_i,\varphi\rangle \right|\leq C \lambda ||\varphi||_1.$$
\noindent Fix such a vector field $X$, then by $L^q-L^{q'}$ duality,

$$\langle H,X\rangle = \lim_{n\to \infty}\int \left(\sum_{i\leq n} (u-u_{B_i})\cdot (\nabla \chi_i)\cdot X\right)\,d\mu.$$
Here, $\nabla \chi_i(x)\cdot X(x)$ denotes the inner product on the tangent space $T_xM$ defined by the Riemannian metric. Since $\sum_{i\leq n} (u-u_{B_i})\nabla \chi_i$ is a finite sum, it defines a smooth function with compact support inside $\Omega$. Using that $\sum_{m\in\N}\chi_m=\mathbf{1}_\Omega$, we have

$$\sum_{i\leq n} (u-u_{B_i})(\nabla \chi_i)\cdot X=\sum_{m\in\N}\sum_{i\leq n} (u-u_{B_i})(\nabla \chi_i)\cdot  (\chi_mX).$$
Denote $X_m:=\chi_m X$, which has now compact support in $\Omega$. Denote by $I_m$ the set of indices $i$ for which $B_i\cap B_m\neq \emptyset$, which is a finite set of cardinal at most $N$ by (5). Then,

\bean
\sum_{i\leq n} (u-u_{B_i})(\nabla \chi_i)\cdot X_m & = & \sum_{i\in I_m,\,i\leq n}(u-u_{B_i})(\nabla \chi_i)\cdot X_m \\
&= & \sum_{i\in I_m,\,i\leq n} (u-u_{B_m})(\nabla\chi_i)\cdot X_m\\
&&+\sum_{i\in I_m,\,i\leq n} (u_{B_m}-u_{B_i})(\nabla\chi_i)\cdot X_m
\eean
We deal with the first sum in the right hand side: using (6),

\bean
\sum_m \sum_{i\in I_m,\,i\leq n} |u-u_{B_m}|\cdot |\nabla\chi_i|\cdot |X_m| & \leq & \sum_m \sum_{i\in I_m} |u-u_{B_m}|\cdot |\nabla\chi_i|\cdot |X_m|\\
& \lesssim & \sum_m \sum_{i\in I_m} \frac{1}{r_i}|u-u_{B_m}|\chi_m\cdot |X|\\
&\lesssim & N\sum_m \frac{1}{r_m}|u-u_{B_m}|\chi_m\cdot |X|
\eean
Integrating the above inequality, using Fubini-Tonelli and Poincar\'e on each ball $B_m$ and (5), we get

\bean
\int \sum_{\substack{
						(m,i)\in\N^2\\
						i\in I_m,\, i\leq n}}						
		 |u-u_{B_m}|\cdot |\nabla\chi_i|\cdot |X_m|\,d\mu & \leq & N \sum_m \int \frac{1}{r_m} |u-u_{B_m}|\cdot \chi_m \cdot  |X|\,d\mu \\ 
&\leq & N \sum_m ||r_m^{-1}(u-u_{B_m}))||_{L^q(B_m)}||X||_{L^{q'}}\\
&\lesssim & N \sum_m ||\nabla u||_{L^q(B_m)}||X||_{L^{q'}}\\
&\lesssim  & N^2 ||\nabla u||_{L^q}\cdot ||X||_{L^{q'}}<+\infty 
\eean
Since the constant $N^2$ is independant of $n$, the limit 

$$\sum_{\substack{
						(m,i)\in\N^2\\
						i\in I_m}}	 (u-u_{B_m})(\nabla \chi_i)\cdot X_m=\lim_{n\to\infty}\sum_{\substack{
						(m,i)\in\N^2\\
						i\in I_m,\, i\leq n}}	 (u-u_{B_m})(\nabla \chi_i) \cdot X_m$$
exists in $L^1$, and and one can evaluate it using Fubini to exchange the order of summation; since $\sum_{i\in\N}\chi_i=\mathbf{1}_\Omega$, we have by definition of $I_m$ that for every $m\in \N$, $\sum_{i\in I_m}\chi_i=1$ in restriction to $B_m$. Thus,

\bean
\sum_{i\in I_m} (\nabla \chi_i )\,\chi_m &=& \chi_m\nabla \left(\sum_{i\in I_m} \chi_i\right)\\
&=&\chi_m (\nabla \mathbf{1})\\
&=& 0.
\eean
Therefore, we have the following equality which holds in $L^1$:

$$\sum_{\substack{
						(m,i)\in\N^2\\
						i\in I_m}}	 (u-u_{B_m})(\nabla \chi_i)\cdot (\chi_m X)=0.$$
Consequently, we have

$$-\langle H,X\rangle =\lim_{n\to\infty} \int \left(\sum_{\substack{(m,i)\in\N^2\\ i\in I_m,\,i\leq n}} (u_{B_m}-u_{B_i})(\nabla \chi_i)\cdot X_m\right)\,d\mu .$$
We now estimate $|u_{B_m}-u_{B_i}|$ for $i\in I_m$. According to (6), $r_i\leq 3 r_m$, and since $B_i\cap B_m\neq \emptyset$, we have $B_i\subset 7B_m$. Also, since $B_m\subset 2B$, the $L^q$ Poincar\'e inequality holds for $7B_m$. We now estimate

\bean
|u_{7B_m}-u_{B_i}| & \leq & \int_{B_i} |u(x)-u_{7B_m}|\,\frac{d\mu(x)}{V(B_i)}\\
&\leq &  \left(\int_{B_i} |u(x)-u_{7B_m}|^q\,\frac{d\mu(x)}{V(B_i)}\right)^{1/q}\\
&\leq &  \left(\int_{7B_m} |u(x)-u_{7B_m}|^q\,\frac{d\mu(x)}{V(B_i)}\right)^{1/q}\\
&\lesssim &  r_m\left(\int_{7B_m} |\nabla u(x)|^q\,\frac{d\mu(x)}{V(B_m)}\right)^{1/q},
\eean 
where in the last line we have used doubling and the fact that $r_i\simeq r_m$. A completely analogous argument gives

$$|u_{7B_m}-u_{B_m}| \lesssim  r_m\left(\int_{7B_m} |\nabla u(x)|^q\,\frac{d\mu(x)}{V(B_m)}\right)^{1/q},$$
and summing these two estimates we find that

$$|u_{B_m}-u_{B_i}| \lesssim r_m\left(\int_{7B_m} |\nabla u(x)|^q\,\frac{d\mu(x)}{V(B_m)}\right)^{1/q}.$$
Given that $|\nabla \chi_i|\lesssim \frac{1}{r_i}\lesssim \frac{1}{r_m}$, one obtains

$$|u_{B_m}-u_{B_i}| \cdot |\nabla \chi_i|\lesssim \left(\int_{7B_m} |\nabla u(x)|^q\,\frac{d\mu(x)}{V(B_m)}\right)^{1/q}.$$
However, (7) entails that $7B_m\cap F\neq \emptyset$. The definition of $F$ in terms of the maximal function gives that

$$\int_{7B_m} |\nabla u(x)|^q\,\frac{d\mu(x)}{V(B_m)}\leq \lambda^q,$$
hence

$$|u_{B_m}-u_{B_i}| \cdot |\nabla \chi_i|\lesssim \lambda.$$
It follows that

\bean
\sum_{\substack{(m,i)\in\N^2\\ i\in I_m,\,i\leq n}} |u_{B_m}-u_{B_i}|\cdot |\nabla \chi_i|\cdot |X_m| &\lesssim & N\lambda \left(\sum_m\chi_m \right)|X|\\
& \lesssim & \lambda |X|.
\eean
Therefore, integrating one finds

$$\int \sum_{\substack{(m,i)\in\N^2\\ i\in I_m,\,i\leq n}} |u_{B_m}-u_{B_i}|\cdot |\nabla \chi_i|\cdot |X_m|d\mu \lesssim \lambda ||X||_1.$$
One thus concludes that

$$\int \sum_{\substack{(m,i)\in\N^2\\ i\in I_m}} |u_{B_m}-u_{B_i}|\cdot |\nabla \chi_i|\cdot |X_m| d\mu \lesssim \lambda ||X||_1.$$
But one has

$$|\langle H,X\rangle| \leq \int \sum_{\substack{(m,i)\in\N^2\\ i\in I_m}} |u_{B_m}-u_{B_i}|\cdot |\nabla \chi_i|\cdot |X_m| d\mu,$$
which finally yields

$$|\langle H,X\rangle|\lesssim \lambda ||X||_1.$$
This estimate being valid for every $X\in L^1(TM)\cap L^{q'}(TM)$, which is dense in $L^1(TM)$, we conclude by duality that $H\in L^\infty$ with $||H||_{L^\infty}\lesssim \lambda$.

\begin{Rem}
{\em 
Note that the proof actually yields the following representation for $H$:

$$H=\sum_{\substack{(m,i)\in\N^2\\ i\in I_m}} (u_{B_i}-u_{B_m})\cdot \chi_m\cdot (\nabla \chi_i)\quad \mbox{a.e.},$$
the right-hand side being an essentially bounded vector field with $L^\infty$ norm bounded by $C\lambda$ for some $C>0$.

}
\end{Rem}

\begin{center}{\bf Acknowledgements} \end{center}
We warmly thank A. Sikora for having pointed out a mistake in the first (pre-published) version of this paper.

This work was partly supported by the French ANR project
RAGE ANR-18-CE40-0012. B. Devyver was also supported in the framework of the “Investissements d’avenir” program
(ANR-15-IDEX-02) and the LabEx PERSYVAL (ANR-11-LABX-0025-01)

\bibliographystyle{plain}               
\bibliography{reverse-riesz}

\begin{thebibliography}{10}

\bibitem{A}
P.~{Auscher}.
\newblock {On the Calder\'on-Zygmund lemma for Sobolev functions}.
\newblock {\em {preprint arXiv:0810.5029}}, 2008.

\bibitem{AC}
P.~{Auscher} and T.~{Coulhon}.
\newblock {Riesz transform on manifolds and Poincar\'e inequalities}.
\newblock {\em {Ann. Sc. Norm. Super. Pisa, Cl. Sci. (5)}}, 4(3):531--555,
  2005.

\bibitem{B}
D.~{Bakry}.
\newblock {The Riesz transforms associated with second order differential
  operators}.
\newblock {Stochastic processes, Proc. 8th Semin., Gainesville/Florida 1988,
  Prog. Probab. 17, 1-43 (1989).}, 1989.

\bibitem{Bu}
P.~Buser.
\newblock A note on the isoperimetric constant.
\newblock {\em Ann. Sci. {\'E}c. Norm. Sup{\'e}r. (4)}, 15:213--230, 1982.

\bibitem{C}
G.~{Carron}.
\newblock {Riesz transform on manifolds with quadratic curvature decay}.
\newblock {\em {Rev. Mat. Iberoam.}}, 33(3):749--788, 2017.

\bibitem{CCH}
G.~Carron, T.~Coulhon, and A.~Hassell.
\newblock {Riesz transform and $L^p$-cohomology for manifolds with Euclidean
  ends}.
\newblock {\em Duke Math. J.}, 133(1):59--94, 2006.

\bibitem{Meyer}
R.~Coifmann and Y.~{Meyer}.
\newblock {\em Wavelets -- Calder\'on-Zygmund and multilinear operators},
  volume~48 of {\em Camb. Stud. Adv. Math.}
\newblock Cambridge: Cambridge University Press, 1997.

\bibitem{CD}
T.~{Coulhon} and X.T. {Duong}.
\newblock Riesz transforms for {{\(1\leq p\leq 2\)}}.
\newblock {\em Trans. Am. Math. Soc.}, 351(3):1151--1169, 1999.

\bibitem{Da}
E.~B. {Davies}.
\newblock {Non-Gaussian aspects of heat kernel behaviour}.
\newblock {\em {J. Lond. Math. Soc., II. Ser.}}, 55(1):105--125, 1997.

\bibitem{BD}
B.~{Devyver}.
\newblock {A perturbation result for the Riesz transform}.
\newblock {\em {Ann. Sc. Norm. Super. Pisa, Cl. Sci. (5)}}, 14(3):937--964,
  2015.

\bibitem{DR}
B.~Devyver and E.~Russ.
\newblock Hardy spaces on {R}iemannian manifolds with quadratic curvature
  decay.
\newblock {\em Anal. PDE}, 15(5):1169--1213, 2022.

\bibitem{GI}
A.~{Grigor'yan} and S.~{Ishiwata}.
\newblock {Heat kernel estimates on a connected sum of two copies of
  \(\mathbb{R}^n\) along a surface of revolution}.
\newblock {\em {Glob. Stoch. Anal.}}, 2(1):29--65, 2012.

\bibitem{GS}
A.~{Grigor'yan} and L.~{Saloff-Coste}.
\newblock {Stability results for Harnack inequalities.}
\newblock {\em {Ann. Inst. Fourier}}, 55(3):825--890, 2005.

\bibitem{GH1}
C.~{Guillarmou} and A.~{Hassell}.
\newblock {Resolvent at low energy and Riesz transform for Schr\"odinger
  operators on asymptotically conic manifolds. I.}
\newblock {\em {Math. Ann.}}, 341(4):859--896, 2008.

\bibitem{HK}
P.~{Haj{\l}asz} and P.~{Koskela}.
\newblock {\em {Sobolev met Poincar\'e}}, volume 688.
\newblock Providence, RI: American Mathematical Society (AMS), 2000.

\bibitem{He}
D.~{He}.
\newblock {Riesz and reverse Riesz on Manifolds with Quadratically Decaying
  Curvature}.
\newblock {Preprint https://arxiv.org/abs/2503.14864v1}, 2025.

\bibitem{J}
R.~{Jiang}.
\newblock {Riesz transform via heat kernel and harmonic functions on
  non-compact manifolds}.
\newblock {\em {Adv. Math.}}, 377:51, 2021.
\newblock Id/No 107464.

\bibitem{K}
S.~{Keith} and X.~{Zhong}.
\newblock {The Poincar\'e inequality is an open ended condition}.
\newblock {\em {Ann. Math. (2)}}, 167(2):575--599, 2008.

\bibitem{L}
M~Lansade.
\newblock Lower bound of {S}chr\"odinger operators on {R}iemannian manifolds.
\newblock {\em arXiv preprint arXiv:2012.08841}, 2020.

\bibitem{Min}
V.~{Minerbe}.
\newblock {Weighted Sobolev inequalities and Ricci flat manifolds}.
\newblock {\em {Geom. Funct. Anal.}}, 18(5):1696--1749, 2009.

\bibitem{Saloffbook}
L.~Saloff-Coste.
\newblock {\em Aspects of {S}obolev-type inequalities}, volume 289.
\newblock Cambridge University Press, 2002.

\bibitem{topics}
E.~M. {Stein}.
\newblock {\em {Topics in harmonic analysis related to the Littlewood-Paley
  theory}}, volume~63.
\newblock Princeton University Press, Princeton, NJ, 1970.

\end{thebibliography}

\end{document}